\documentclass[11pt]{article}
\usepackage[margin=1in]{geometry}

\usepackage[T1]{fontenc}
\usepackage[utf8]{inputenc}

\usepackage{comment}
\usepackage{enumitem}
\usepackage{hyperref}
\usepackage{xcolor}
\usepackage{xspace}
\usepackage[most]{tcolorbox}
\usepackage{cleveref}
\usepackage{authblk}
\usepackage{accents}

\usepackage{amsthm}
\usepackage{amsmath}
\usepackage{amssymb}
\usepackage{mathtools}

\usepackage{thmtools}
\usepackage{thm-restate}

\usepackage{float}
\usepackage{graphicx}
\usepackage{epstopdf} 

    \usepackage{tikz}
    \usetikzlibrary{arrows}
    \usetikzlibrary{calc}
    \usetikzlibrary{positioning}

\usepackage{nicefrac}

\newcommand{\say}[1]{``#1''}
\newcommand{\tsc}[1]{\textsc{\text{#1}}}

\newenvironment{todo}[1]{{\bf \textcolor{yellow!70!green}{Todo:~#1}}}{}

\declaretheorem[name=Theorem]{theorem}
\declaretheorem[name=Corollary,numberlike=theorem]{corollary}
\declaretheorem[name=Lemma,numberlike=theorem]{lemma}
\declaretheorem[name=Observation,numberlike=theorem]{observation}
\crefname{observation}{observation}{observations}
\Crefname{observation}{Observation}{Observations}

\crefname{proposition}{proposition}{propositions}
\Crefname{proposition}{Proposition}{Propositions}
\declaretheorem[name=Conjecture,numberlike=theorem]{conjecture}
\crefname{conjecture}{conjecture}{conjectures}
\Crefname{conjecture}{Conjecture}{Conjectures}
\declaretheorem[name=Question,numberlike=theorem]{question}

\crefname{question}{question}{questions}
\Crefname{question}{Question}{Questions}
\declaretheorem[name=Definition,style=definition,numberlike=theorem]{definition}
\crefname{definition}{definition}{definitions}
\Crefname{definition}{Definition}{Definitions}
\declaretheorem[name=Claim,style=remark]{claim}
\crefname{claim}{claim}{claims}
\Crefname{Claim}{Claim}{Claims}

\renewenvironment{proof}{\par \noindent \textit{Proof.} }{\hfill$\Box$\\}
\newenvironment{proofof}[1]{\par \noindent \textit{Proof of #1.}
}{\hfill$\Box$\medskip}
\newcommand{\mydiamond}{\rotatebox[origin=c]{45}{$\vcenter{\hbox{$\Box$}}$}}
\newenvironment{subproof}[1]{\par\noindent \textit{Proof of #1.}\ }{\hfill \mydiamond \par\vspace{11pt}}

\crefname{subsection}{subsection}{subsections}
\Crefname{subsection}{Subsection}{Subsections}
\crefname{subsubsection}{subsubsection}{subsubsections}
\Crefname{subsubsection}{Subsubsection}{Subsubsections}

    \newcommand{\defproblem}[3]{
    \vspace{1mm}
    \noindent\fbox{
    \begin{minipage}{0.96\textwidth}
    \begin{tabular*}{\textwidth}{@{\extracolsep{\fill}}lr} #1 \\ \end{tabular*}
    {\textbf{Input:}} #2 \\
    {\textbf{Question:}} #3
    \end{minipage}
    }
    \vspace{1mm}
   }
    \newcommand{\defparproblem}[4]{
     \vspace{1mm}
     \noindent\fbox{
     \begin{minipage}{0.96\textwidth}
     \begin{tabular*}{\textwidth}{@{\extracolsep{\fill}}lr} #1 \\ \end{tabular*}
     {\textbf{Parameter:}} #3 \\
     {\textbf{Input:}} #2 \\
     {\textbf{Question:}} #4
     \end{minipage}
     }
     \vspace{1mm}
     \newline
    }

    \newcommand{\ora}[1]{\overrightarrow{#1}}
    \newcommand{\ola}[1]{\overleftarrow{#1}}

    \newcommand{\cF}{\mathcal{F}}

    \newcommand{\cS}{\mathcal{S}}
    \newcommand{\cT}{\mathcal{T}}
    
    \newcommand{\cX}{\mathcal{X}}

    \newcommand{\psca}{\textsc{Plane Strong Connectivity Augmentation}\xspace}

    \newcommand{\sopsca}{\textsc{PSCA}\xspace}
    \newcommand{\sdpsca}{\textsc{DPSCA}\xspace}
    \newcommand{\spsca}{\textsc{PSCA}\xspace}
    \newcommand{\spscap}{\textsc{PSCA$'$}\xspace}

    \newcommand{\dirPSCA}{\text{\textsc{Directed-PSCA}}\xspace}

    \newcommand{\lt}{\mathsf{lt}}
    
    \newcommand{\AF}{\mathcal{AF}} 
    \newcommand{\SF}{\mathcal{SF}} 

\RequirePackage{stmaryrd}
\usepackage{textcomp}
\DeclareUnicodeCharacter{2286}{\subseteq}
\DeclareUnicodeCharacter{2192}{\ifmmode\to\else\textrightarrow\fi}
\DeclareUnicodeCharacter{2203}{\ensuremath\exists}
\DeclareUnicodeCharacter{183}{\cdot}
\DeclareUnicodeCharacter{2200}{\forall}
\DeclareUnicodeCharacter{2264}{\leq}
\DeclareUnicodeCharacter{2265}{\geq}
\DeclareUnicodeCharacter{8614}{\mathbin{\mapsto}}
\DeclareUnicodeCharacter{8656}{\Leftarrow}
\DeclareUnicodeCharacter{8657}{\Uparrow}
\DeclareUnicodeCharacter{8658}{\Rightarrow}
\DeclareUnicodeCharacter{8659}{\Downarrow}
\DeclareUnicodeCharacter{8669}{\rightsquigarrow}
\newcommand{\eqdef}{\stackrel{{\scriptsize\rm def}}{=}}
\DeclareUnicodeCharacter{8797}{\eqdef}
\DeclareUnicodeCharacter{8870}{\vdash}
\DeclareUnicodeCharacter{8873}{\Vdash}
\DeclareUnicodeCharacter{22A7}{\models}
\DeclareUnicodeCharacter{9121}{\lceil}
\DeclareUnicodeCharacter{9123}{\lfloor}
\DeclareUnicodeCharacter{9124}{\rceil}
\DeclareUnicodeCharacter{2208}{\in}
\DeclareUnicodeCharacter{9126}{\rfloor}
\DeclareUnicodeCharacter{9655}{\triangleright}
\DeclareUnicodeCharacter{9665}{\triangleleft}
\DeclareUnicodeCharacter{9671}{\diamond}
\DeclareUnicodeCharacter{9675}{\circ}
\DeclareUnicodeCharacter{10178}{\bot}
\DeclareUnicodeCharacter{10214}{} 
\DeclareUnicodeCharacter{10215}{} 
\DeclareUnicodeCharacter{10229}{\longleftarrow}
\DeclareUnicodeCharacter{10230}{\longrightarrow}
\DeclareUnicodeCharacter{10231}{\longleftrightarrow}
\DeclareUnicodeCharacter{10232}{\Longleftarrow}
\DeclareUnicodeCharacter{10233}{\Longrightarrow}
\DeclareUnicodeCharacter{10234}{\Longleftrightarrow}
\DeclareUnicodeCharacter{10236}{\longmapsto}
\DeclareUnicodeCharacter{10238}{\Longmapsto} 
\DeclareUnicodeCharacter{10503}{\Mapsto}    
\DeclareUnicodeCharacter{10971}{\mathrel{\not\hspace{-0.2em}\cap}}
\DeclareUnicodeCharacter{65294}{\ldotp}
\DeclareUnicodeCharacter{65372}{\mid}

\colorlet{myGreen}{green!50!black}
\colorlet{myLightgreen}{green}
\colorlet{myRed}{red!90!black}
\definecolor{myBlue}{rgb}{0.25, 0.0, 1.0}
\definecolor{myLightBlue}{rgb}{0.39, 0.58, 0.93}
\colorlet{myViolet}{myBlue!55!myRed}
\definecolor{myOrange}{rgb}{1.0, 0.66, 0.07}

\definecolor{CornflowerBlue}{rgb}{0.39, 0.58, 0.93}
\definecolor{DarkGoldenrod}{rgb}{0.72, 0.53, 0.04}
\definecolor{BritishRacingGreen}{rgb}{0.0, 0.26, 0.15}
\definecolor{DarkMagenta}{rgb}{0.55, 0.0, 0.55}
\definecolor{AO}{rgb}{0.0, 0.5, 0.0}
\definecolor{BostonUniversityRed}{rgb}{0.8, 0.0, 0.0}
\definecolor{myRed}{rgb}{0.8, 0.0, 0.0}
\definecolor{DarkMidnightBlue}{rgb}{0.0, 0.2, 0.4}
\definecolor{DarkTangerine}{rgb}{1.0, 0.66, 0.07}
\definecolor{AppleGreen}{rgb}{0.55, 0.71, 0.0}
\definecolor{BrightUbe}{rgb}{0.82, 0.62, 0.91}
\definecolor{Amethyst}{rgb}{0.6, 0.4, 0.8}
\definecolor{DarkGray}{rgb}{0.52, 0.52, 0.51}
\definecolor{Gray}{rgb}{0.66, 0.66, 0.66}
\definecolor{BananaYellow}{rgb}{1.0, 0.88, 0.21}
\definecolor{Amber}{rgb}{1.0, 0.75, 0.0}
\definecolor{LightGray}{rgb}{0.83, 0.83, 0.83}
\definecolor{PrincetonOrange}{rgb}{1.0, 0.56, 0.0}
\definecolor{DeepCarrotOrange}{rgb}{0.91, 0.41, 0.17}
\definecolor{CarrotOrange}{rgb}{0.93, 0.57, 0.13}
\definecolor{MidnightBlue}{rgb}{0.1, 0.1, 0.44}
\definecolor{Magenta}{rgb}{0.50, 0.0, 0.50}
\definecolor{BrightPink}{rgb}{1.0, 0.0, 0.5}
\definecolor{BrilliantRose}{rgb}{1.0, 0.33, 0.64}
\definecolor{ChromeYellow}{rgb}{1.0, 0.65, 0.0}
\definecolor{HotMagenta}{rgb}{1.0, 0.11, 0.81}

\definecolor{DarkTangerine}{rgb}{1.0, 0.66, 0.07}
\definecolor{darkyellow}{rgb}{.7, .6, 0.0}
\definecolor{CornflowerBlue}{rgb}{0.39, 0.58, 0.93}
\definecolor{DarkGoldenrod}{rgb}{0.72, 0.53, 0.04}
\definecolor{BritishRacingGreen}{rgb}{0.0, 0.26, 0.15}
\definecolor{AO}{rgb}{0.0, 0.5, 0.0}
\definecolor{MidnightBlack}{rgb}{0.1,0.1,.34}
\definecolor{MidnightBlue}{rgb}{0.1,0.1,0.43}
\definecolor{Black}{rgb}{0,0, 0}
\definecolor{Blue}{rgb}{0, 0 ,1}
\definecolor{Red}{rgb}{1, 0 ,0}
\definecolor{White}{rgb}{1, 1, 1}
\definecolor{DeepMagenta}{rgb}{0.8, 0.0, 0.8}
\definecolor{grey}{rgb}{.6, .6, .6}
\definecolor{darkgrey}{rgb}{.33, .33, .33}
\definecolor{Mygreen}{rgb}{.0, .7, .0}
\definecolor{Yellow}{rgb}{.55,.55,0}
\definecolor{Mustard}{rgb}{1.0, 0.86, 0.35}
\definecolor{applegreen}{rgb}{0.55, 0.71, 0.0}
\definecolor{darkturquoise}{rgb}{0.0, 0.81, 0.82}
\definecolor{celestialblue}{rgb}{0.29, 0.59, 0.82}
\definecolor{green_yellow}{rgb}{0.68, 1.0, 0.18}
\definecolor{crimsonglory}{rgb}{0.75, 0.0, 0.2}
\definecolor{darkmagenta}{rgb}{0.30, 0.0, 0.30}
\definecolor{magenta}{rgb}{0.50, 0.0, 0.50}
\definecolor{internationalorange}{rgb}{1.0, 0.31, 0.0}
\definecolor{darkorange}{rgb}{1.0, 0.55, 0.0}
\definecolor{ao}{rgb}{0.0, 0.5, 0.0}
\definecolor{awesome}{rgb}{1.0, 0.13, 0.32}
\definecolor{darkcyan}{rgb}{0.0, 0.50, 0.50}
\definecolor{violet}{rgb}{0.93, 0.51, 0.93}
\definecolor{brown}{rgb}{0.65, 0.16, 0.16}
\definecolor{orange}{rgb}{1.0, 0.65, 0.0}
\definecolor{DarkGreen}{rgb}{0,.5,0}
\definecolor{BostonUniversityRed}{rgb}{0.8, 0.0, 0.0}

\newcommand{\red}[1]{{\color{Red}#1}}

\newcommand{\FPT}{\textsf{FPT}\xspace}
\newcommand{\NP}{\textsf{NP}\xspace}
\newcommand{\yes}{{\sf yes}\xspace}
\newcommand{\no}{{\sf no}\xspace}

\newcommand{\Xend}{\accentset{\circ}{X}}
\title{Plane Strong Connectivity Augmentation
\thanks{SB and DT were supported by the ANR project GODASse ANR-24-CE48-4377.
AR was supported by the Polish National
Science Centre SONATA BIS-12 grant number 2022/46/E/ST6/00143. DT was supported by the French-German Collaboration ANR/DFG Project UTMA (ANR-20-CE92-0027), the  and the Franco-Norwegian project PHC AURORA 2024-2025 (Projet n°\! 51260WL).}}

\author[1]{Stéphane Bessy}
\author[1]{Daniel Gonçalves}
\author[1,2]{Amadeus Reinald}
\author[1]{Dimitrios M. Thilikos}
\affil[1]{LIRMM, Univ Montpellier, CNRS, Montpellier, France}
\affil[2]{University of Warsaw, Poland}

\begin{document}

\date{}

\maketitle

\begin{abstract}

    \noindent  We investigate the problem of strong connectivity augmentation within plane oriented graphs.
    We show that deciding whether a plane oriented graph $D$ can be augmented with (any number of) arcs $X$ such that $D+X$ is strongly connected, but still plane and oriented, is \NP-hard.
    This question becomes trivial within plane \textsl{di}graphs, like most connectivity augmentation problems without a budget constraint.
    
    The budgeted version, \tsc{Plane} \tsc{Strong} \tsc{Connectivity} \tsc{Augmentation} \tsc{(PSCA)} considers a plane oriented graph $D$ along with some integer $k$, and asks for an $X$ of size at most $k$ ensuring that $D+X$ is strongly connected, while remaining plane and oriented.
    Our main result is a fixed-parameter tractable algorithm for \spsca, running in time $2^{O(k)} n^{O(1)}$.
    The cornerstone of our procedure is a structural result showing that, for any fixed $k$, each face admits a bounded number of partial solutions \say{dominating} all others.
    Then, our algorithm for \spsca combines face-wise branching with a Monte-Carlo reduction to the polynomial \tsc{Minimum Dijoin} problem, which we derandomize.
    To the best of our knowledge, this is the first \FPT algorithm for a (hard) connectivity augmentation problem constrained by planarity.
\end{abstract}


\section{Introduction}

Suppose you are given a non-crossing network of one-way roads between cities, but where some cities are not reachable from others. The problem we address in this paper is the following:
\begin{quote}
    \textsl{How to construct a set of new one-way roads ensuring that every city is reachable from every other, without introducing crossings or two-way roads?}
\end{quote}
The above can be formalized as a problem over plane (embedded) oriented graphs (digraphs forbidding cycles of length two).
Then, we are asking for an \emph{augmentation} (addition of arcs) achieving strong connectivity, while maintaining a plane oriented graph.
The first question of interest to us is \say{augmentability}, that is whether such an augmentation exists at all, then we address the computation of a minimum augmentation.
The latter minimization question belongs to the broader class of connectivity augmentation problems, a well-established area of study in both undirected and directed graphs, with applications to Survivable Network Design.
For a comprehensive overview, see the survey chapter by Frank and Jordán~\cite{frankGraphConnectivityAugmentation2015}.

Connectivity augmentation has a long history in the undirected setting, mainly as minimization problems, rather than augmentability ones.
Indeed, most augmentability questions become trivial without a budget condition: simply add all possible edges (arcs) and verify the resulting connectivity.
The most studied examples are \tsc{$c$-Vertex} \tsc{Connectivity} \tsc{Augmentation} and \tsc{$c$-Edge} \tsc{Connectivity} \tsc{Augmentation}, respectively $c$-\tsc{VCA} and $c$-\tsc{ECA}.
These problems ask for the minimum number of edge additions —allowing parallel edges— required to make a graph $c$-vertex-connected or $c$-edge-connected respectively.
For edge-connectivity, a uniform polynomial-time algorithm for $c$-\tsc{ECA} was established by Watanabe and Nakamura~\cite{watanabeEdgeconnectivityAugmentationProblems1987}.
In contrast, the complexity of $c$-\tsc{VCA} stands as a major open problem in the field, though fixed-parameter tractability in $c$ has been established by Jackson and Jordán~\cite{jacksonIndependenceFreeGraphs2005}.
Weighted versions of these problems are already \NP-hard for $c = 2,$ via simple reductions from \tsc{Hamiltonian Cycle}~\cite{eswaranAugmentationProblems1976}. While constant-factor approximation algorithms are known~\cite{fredericksonApproximationAlgorithmsSeveral1981, traub20231}, the parameterized complexity of the weighted variants remains largely unresolved (see~\cite{nutovParameterizedNodeConnectivityAugmentation2024}).

For digraphs, the first notion of interest is strong connectivity, asking for a directed path between any two vertices.
In their seminal paper, Eswaran and Tarjan~\cite{eswaranAugmentationProblems1976} showed a polynomial-time algorithm for \tsc{Strong} \tsc{Connectivity} \tsc{Augmentation} (\tsc{SCA}), which asks for a minimal number of arc additions rendering a digraph strongly connected.
For higher connectivities, \tsc{$k$-Arc SCA} was shown to be (uniformly) polynomial by Frank~\cite{frankAugmentingGraphsMeet1992}.  
As in the undirected case, the complexity of \tsc{$k$-Vertex SCA} is yet to be settled, though Frank and Jordan~\cite{frankMinimalEdgeCoveringsPairs1995} showed the problem is \FPT.
As already observed by Eswaran~\cite{eswaranAugmentationProblems1976}, the (arc) weighted versions of these problems are \NP-hard, even for \tsc{Weighted SCA}.
On the positive side, a $2$-approximation was obtained by Frederickson and Ja'Ja~\cite{fredericksonApproximationAlgorithmsSeveral1981}, and more recently, a $2^{O(k\log k)} n^{O(1)}$ \FPT algorithm was established by Klinkby, Misra, and Saurabh~\cite{klinkbyStrongConnectivityAugmentation2021}.

In this paper, we are interested in strong connectivity augmentation constrained by planarity, but also simplicity: oriented graphs forbid parallel arcs.
As we survey below, both of these constraints have received considerable attention in the field, but with only a few positive results.

\paragraph{Simplicity constraints in (di)graphs.}
At the source of the two main polynomial algorithms discussed above, for \tsc{$k$-ECA} and \tsc{$k$-Arc SCA}, is the ability to add parallel edges and arcs (and digons).
In other words, these algorithms are really solving problems within multi(di)graphs, rather than within (simple) graphs.
This may also explain the discrepancy between these positive results for edge (arc) connectivity, and the elusiveness of their vertex-connectivity counterparts.
This was already noted by Frank~\cite{frankGraphConnectivityAugmentation2015}, raising the question of algorithms for connectivity augmentation preserving simplicity.

In the undirected setting, Jordan~\cite{jordanTwoNPCompleteAugmentation1997} showed that the simplicity-preserving variant of \tsc{ECA} is \NP-hard, when the target connectivity $c$ is part of the input.
Still, this problem is \FPT in $c,$ as shown by Bang-Jensen and Jordan~\cite{bang-jensenEdgeConnectivityAugmentationPartition1999} (see also~\cite{johansenEdgeConnectivityAugmentationSimple2025} for an \say{incremental} variant).
In the directed setting, these questions were investigated by Bérczi and Frank~\cite{bercziSupermodularity2018}. There, they obtained a polynomial-time algorithm for the problem of increasing arc-connectivity by one within simple digraphs, forbidding multi-arcs but still allowing digons.

\paragraph{Planarity constraints in graphs.}
Given a (simple) graph, the next natural constraint to impose on the structure of the augmented graph is planarity, which can be understood in multiple settings. 
In the abstract graph setting, corresponding to \emph{planar} problems, both the input and output graphs should be planar, but are not constrained by any embedding.
The topological setting, which is our main interest, corresponds to \emph{plane} versions, where the input comes with a fixed embedding, that must be preserved by the augmentation.
A third setting of interest is given by \emph{geometric} variants, which consider plane straight-line graphs, and ask for augmentations with straight line edges. 

In 1989, Rappaport~\cite{rappaportComputingSimpleCircuits1989} showed the NP-hardness of \tsc{Simple Circuit}.
There, we are given a set of segments embedded in the plane, and ask whether these can be augmented into a simple polygon by the addition of new segments.
This problem can be seen as a restriction of \tsc{$2$-Connectivity Augmentation} within plane straight-line graphs, and appears to be the first consideration of a (connectivity) augmentation problem constrained by planarity.
This can also be seen as an \say{augmentability} problem, because the restriction of being a polygon already imposes that we add exactly as many segments as we were given.
The first \textsl{planar} connectivity augmentation problem that was shown to be NP-hard is \tsc{Planar $2$-VCA}, due to a result of Kant and Bodlaender~\cite{kantPlanarGraphAugmentation1991}.
Rutter and Wolff \cite{rutter_wolff_2012} showed the \NP-hardness for the plane and geometric variants of \tsc{$2$-ECA}, and Gutwenger, Mutzel and Zey~\cite{gutwengerPlanarBiconnectivityAugmentation2009} proved \NP-hardness for \tsc{Plane $2$-VCA}.
On the positive side, \tsc{Planar $2$-VCA} admits a $2$-approximation~\cite{kantPlanarGraphAugmentation1991} algorithm, and \tsc{plane $2$-VCA} is polynomial for connected instances~\cite{gutwengerPlanarBiconnectivityAugmentation2009}. 
Until this paper, no \FPT algorithms have been obtained for \NP-hard connectivity augmentation problems constrained by planarity.

\paragraph{Structural constraints for \tsc{SCA}.}
Today, little is known about the complexity of strong connectivity augmentation subject to structural constraints, other than the simplicity considerations of~\cite{bercziSupermodularity2018}.
Motivated by a problem about the rigidity of square grid frameworks, Gabow and Jordán \cite{gabowHowMakeSquare2000} devised a polynomial-time algorithm for \tsc{SCA} when the input is a bipartite digraph and the bipartition of the input must be preserved.
Another interesting structural restriction, orthogonal to our oriented requirement, is the case where the only new allowed arcs are those $(v,u)$ such that $(u,v)$ belongs to the initial digraph.
That is, we wish to make a digraph strongly connected by turning arcs into digons.
This is in fact equivalent to the \tsc{Minimum Dijoin} problem, which was shown to be polynomial by Frank~\cite{frankHowMakeDigraph1981}, through a constructive proof of the Lucchesi-Younger theorem~\cite{lucchesiMinimaxTheoremDirected1978}.

\subsection{Our results}

We are interested in strong connectivity augmentations within plane simple digraphs.
There are two natural ways of asking for a digraph to be simple: forbidding parallel arcs in the same direction defines digraphs, while forbidding \textsl{any} parallel arcs defines oriented graphs.
We will be dealing with the latter interpretation, motivated by the following \emph{augmentability} question: given a plane simple digraph, can it be augmented to strong connectivity, while staying plane and simple?
If we understand simplicity in the digraph sense, the answer is always yes: any maximal set of arcs yields a bidirected triangulation, which is strongly connected.
Therefore, our augmentability question can only be of interest within plane \textsl{oriented} graphs. Indeed, these cannot always be augmented as desired, take for example any oriented triangulation that is not already strongly connected.
Our first result is that deciding this strong connectivity augmentability within plane oriented graphs is already~\NP-hard.
\begin{restatable}{theorem}{npHardOpsca}\label{thm_np_hardness}
    Deciding whether a plane oriented graph admits a strongly connected augmentation that is plane and oriented is \NP-complete. 
    Moreover, under the Exponential Time Hypothesis, this problem does not admit a $2^{o(\sqrt{n})}$-time algorithm.
\end{restatable}
\noindent
This situates our question as one of the few connectivity augmentation problems that is \NP-hard because of structural constraints, rather than budget ones. Another example being~\tsc{Simple Circuit}~\cite{rappaportComputingSimpleCircuits1989}, where structure already forces an exact size for the solution (which is not the case for \spsca). Our reduction can be adapted to produce $3$-connected instances, which are uniquely embeddable, and yield that the planar variant of our question is also hard.
While~\Cref{thm_np_hardness} is reminiscent of \tsc{Plane $2$-VCA} and \tsc{Planar $2$-VCA} being hard, the hardness of the latter problems does rely on a budget.

The budget-constrained version of our question, \spsca defined below, is also \NP-complete thanks to~\Cref{thm_np_hardness}.

\defparproblem{\psca (\textsc{PSCA})}{A connected plane oriented graph $D$ and  an integer $k$}{$k$}{Is there some $X \subseteq V(D)^2$ with $|X| \leq k,$ such that $D+X$ is strongly connected, oriented, and plane for the same embedding of $D$?}
\medskip

\noindent
Our main result is the fixed-parameter tractability of \spsca parameterized by $k$.
\begin{restatable}{theorem}{fptPsca}\label{thm_fpt}
    \psca is \FPT with respect to the solution size $k$ and admits a $2^{O(k)} n^{O(1)}$ algorithm.
\end{restatable}
\noindent
Our algorithm can be extended to work for any connected oriented graph (cellularly) embedded on a surface of bounded genus.
It would also be interesting to obtain an algorithm for the planar variant of \spsca, where the embedding is not fixed.
Regarding simplicity, it is not clear whether \tsc{SCA} remains polynomial when we impose the resulting digraph to be oriented (as in \spsca).

As we have already argued, the augmentability variant of \spsca is trivial for digraphs, but its budget-constrained version remains of interest:\medskip

\defproblem{\tsc{Directed Plane Strong Connectivity Augmentation} (\dirPSCA)}{A connected plane digraph $D$ and an integer $k$}{Is there some $X \subseteq V(D)^2$ with $|X|\leq k$ such that $D+X$ is strongly connected, directed, and plane for the same embedding of $D$?}\medskip\medskip

\noindent
While we show this problem is \FPT as an introduction to our algorithm for \spsca (\Cref{sec_dirPSCA}), we conjecture it should even be polynomial.
\begin{conjecture}
    \dirPSCA is polynomial-time solvable.
\end{conjecture}

Let us also recall that the parameterized complexity of planar connectivity augmentation problems is wide open. In particular, we lack an undirected analogue of~\Cref{thm_fpt} for biconnectivity.
\begin{question}
    Is \tsc{Plane $2$-VCA} \FPT parameterized by solution size?
\end{question}

\paragraph{Scheme of the \FPT algorithm.}
Our strategy to solve \spsca (and \dirPSCA) proceeds by searching for \emph{supported} solutions.
Intuitively, these correspond to solutions $X$ where in each face $F$ of $D$, the endpoints of $X$ are skewed towards \say{terminal components} of the digraph induced by $F$.
While defining supported solution is already technical, our main structural result is a bound on the number of supported completions — restrictions of a supported solution to $F$ — as a function of $k$.
Then, we distinguish simple and alternating faces, according to the number of \say{terminal} components they induce.
We show that for positive instances $(D,k)$, the number of alternating faces is always bounded by a function of $k$, which enables branching over all supported completions over them.
This reduces our question to the computation of a solution adding arcs only within simple faces.
The main obstacle then is that the number of simple faces can be linear in the size of the instance, but crucially, each of them admits a (absolute) constant number of \say{minimal} supported completions.
We devise a randomized procedure guessing one \say{allowed} supported completion within each simple face, then reduce the problem of finding a minimum solution using only those arcs to \tsc{Minimum Dijoin}. Using the algorithm of Frank~\cite{frankHowMakeDigraph1981} for the latter yields a polynomial Monte-Carlo algorithm computing a solution within simple faces. This can then be derandomized using universal sets, at the cost only of a $2^{O(k)} n^{O(1)}$ factor to the running time.
Combining our branching for alternating faces with the (derandomized) algorithm for simple faces yields a $2^{O(k)} n^{O(1)}$ algorithm deciding the existence of a (supported) solution for \spsca.

\paragraph{Structure of the paper.}
After some preliminaries in~\Cref{sec_def}, we show an \FPT algorithm for \dirPSCA in~\Cref{sec_dirPSCA}, serving as an introduction to our algorithm for~\spsca.
While the overarching steps are common, most technicalities of \spsca are short-circuited in~\dirPSCA, where supported solutions admit a very natural definition.
We finish this section with~\Cref{ssec_adapt}, discussing the technicalities of \spsca compared to \dirPSCA: defining supported solutions, and the necessity to randomize the computation of a solution within simple faces.
Then~\Cref{sec_psca} shows shows our \FPT algorithm for \spsca, achieving to prove~\Cref{thm_fpt}.
Finally, we show the hardness of \spsca in \Cref{sec_hardness}, yielding \Cref{thm_np_hardness}.

\section{Preliminaries}\label{sec_def}


All digraphs considered in this paper are plane and given with their embedding.
We identify any face $F$ of a plane digraph $D$ with its \emph{boundary}, defined as the cyclic list of labelled vertices $(v_1,...,v_r)$ following the closed walk around $F$. When $D$ is not $3$-connected, note that the same vertex may appear multiple times along the boundary.
For any digraph $D$ and any $U \subseteq V$, $D[U]$ is the subdigraph induced by $U$.

\paragraph{Connectivity.}
Throughout the paper, we use \say{strong} and \say{strongly connected} interchangeably, and often refer to directed paths as \emph{dipaths}.
A dicut in a digraph $D = (V,A)$ is a bipartition $(Z,V \setminus Z)$ such that all arcs across are directed from $Z$ to $V \setminus Z$.
A strong component is \emph{trivial} if it consists of a single vertex.
A strong component $C$, in particular a vertex, is a {\em source} if $(C,V \setminus C)$ is a dicut, and a {\em sink} if $(V \setminus C,C)$ is. 
Sources and sinks of $D$ form the set $\cT(D)$ of {\em terminal components}.
Then, a \emph{condensation} of $D$ is a (plane) digraph obtained by identifying each strong component into a single vertex, and removing copies of the same arc.

\paragraph{Completions and their restrictions.}
For a plane digraph $D=(V,A)$, and any $Y \subseteq V(D)^2$ given with an embedding, both $D+Y = (V,A \cup Y)$ and $D-Y = (V, A \setminus Y)$ are plane digraphs respecting the embedding of $D$.
A \emph{completion} of $D$ is a subset $X$ of $V(D)^2$, given with an embedding, such that $D+X$ is plane and oriented (directed in \Cref{sec_dirPSCA}).
When $D+X$ is moreover strong, $X$ is a \emph{solution} of $D$.
Given a face $F$, a completion of $F$ is a completion of $D$ such that $X' \subseteq (V(F))^2,$ and all arcs of $X'$ are embedded in $F.$
For any completion $X$, and any face $F$ of $D$, the \emph{restriction} $X_F$ of $X$ to $F$ as the subset of $X$ embedded in $F$.
For a subwalk $I$ of $F$, the \emph{restriction} $X_I$ of $X$ to $I$ is the set of arcs in $X$ having at least one endpoint on $I$.

\section{Fixed-parameter tractability when allowing digons}\label{sec_dirPSCA}

In this section, we describe an \FPT algorithm for \dirPSCA, taking instances $(D,k)$ with $D$ a plane digraph, and allowing digons in the solution.
Throughout this section, we thus consider \emph{completions} (and \emph{solutions}) as some $X$ of $V(D)^2$ such that $D+X$ is a plane \textsl{di}graph.
We begin by describing the structure of our instance, local terminals, as well as simple and alternating faces in~\Cref{ssec_dir_structure_instance}.
Then, we show the existence of supported solutions in~\Cref{ssec_dir_struct-sol}. We bound the number of possible supported completions across all alternating faces in~\Cref{ssec_dir-alternating-faces}, allowing us to branch over them.
Then, in~\Cref{ssec_dir_FPT}, we finish the description of the \FPT algorithm, which reduces to computing a minimum (supported) solution within simple faces.

\subsection{Structure of the instance}\label{ssec_dir_structure_instance}

Our first step is to reduce any instance of \dirPSCA to (polynomially many) acyclic ones.
\begin{restatable}{lemma}{reducDAG}\label{lem_reduc-DAG}
    \dirPSCA admits a polynomial-time Turing reduction to the same problem over acyclic instances.
\end{restatable}
\noindent
Such a reduction can easily be obtained by contracting the arcs belonging to a strong component until we reach a plane condensation of the digraph at hand. However, this involves some non-essential technicalities when contracting loops, and we delay the proof to~\Cref{apdx:reduc-DAG}.
In all of the following, we may thus assume $(D,k)$ is an acyclic instance of \dirPSCA.

\paragraph{Local terminals and face types.} 
For any face $F$, a \textsl{labelled} vertex $v_i$ of $F$ is a {\em local source} if its incident arcs around the boundary are both outgoing, and a {\em local sink} if they are both incoming.
Local sources and sinks form the \emph{local terminals} of $F$, see \Cref{fig:simple-alt-directed}.
For any face $F,$ local sources and local sinks must appear in alternation along the boundary.
Then, any consecutive local terminals are connected by a directed walk, which—due to acyclicity—is in fact a directed path.
Therefore, the number of local sources in $F$ equals the number of its local sinks, and we define $\lt(F)$ to be the (even) number of local terminals of $F$.
Note that $lt(F)$ is always non-zero for acyclic instances, for otherwise the boundary of $F$ would form a directed cycle.
For the remainder of this section, we say that $F$ is a {\em simple face} if it contains exactly one local source and one local sink; otherwise, it is called an {\em alternating face}, see \Cref{fig:simple-alt-directed} for a depiction.

\begin{figure}
    \centering
    \includegraphics[scale=1.2,page=2]{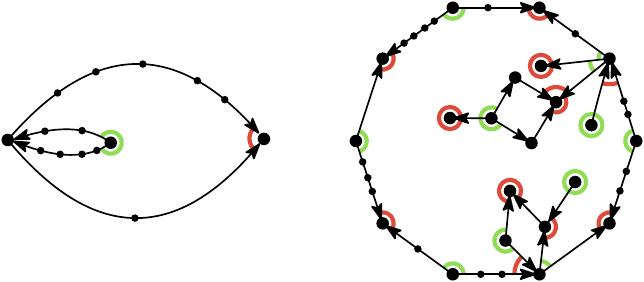}
    \caption{
        A simple face, with two local terminals and  an alternating face with $22$ local terminals. Local sources are shown as green angles, while local sinks are the red angles. Note that the same vertex can appear in both local sinks, local sources and even non-terminal angles.}
        \label{fig:simple-alt-directed}
\end{figure}

\subsection{Existence of supported solutions}\label{ssec_dir_struct-sol}

Throughout the section, a completion of a face $F$ is \emph{supported} if all of its arcs lie between local terminals of $F$.
Then, a completion (or solution) $X$ of $D$ is {\em supported} if the restriction of $X$ to each face is supported.
The following lemma shows that any solution reconfigures into a supported one by a sequence of appropriate \say{shifts} of its endpoints towards local terminals.
\begin{lemma}\label{lem_supported-sol-dir-psca}    
If an acyclic instance of \dirPSCA admits a solution of size $k,$ then it admits a supported solution of size at most $k.$
\end{lemma}
\begin{proof}
    Consider an instance $(D,k)$, with $D$ acyclic, and assume that it admits a solution $X$ of size at most $k$.
    To reconfigure $X$ into a supported solution, we iterate over all faces $F$, and reconfigure $X_F$ in $X$ into a supported completion, while maintaining that $X$ is a solution.
    Inside each face $F$, we successively consider any subwalk $P$ from a local source $s$ to a consecutive local sink $t$ (recall $P$ is then a directed path).
    We will show how to reconfigure the endpoints of the restriction $X_P$ belonging to $P$, and only those, to be incident only to $s$ and $t$.
    It suffices to show this last step, which iterated as above for each pair of consecutive local terminals, and each face, will achieve to reconfigure $X$ into a supported solution.
    In the following, let us consider such a $P$ between $s$ and $t$, and show the reconfiguration of $X_P$.
    
    We first consider the case where $X_P$ contains an arc with both endpoints on $P$. We modify the endpoints of $X_P$ on $P$ as follows.
    Remove all arcs with both endpoints on $P$ and replace them with (a single) $(t,s)$. Then for all other arcs of $X_P$, substitute all their endpoints in $P$ with $s$.
    It is easy to see that these arcs can be modified while preserving a plane embedding for $D+X$.
    The addition of $(t,s)$ ensures that $V(P)$ is now strongly connected in $D+X$.
    Then, consider any arc $(x,p)$ in the initial $X_P$, with $p \in P$ and $x \notin P$. Since its modification yields $(x,s)$, note that the path from $x$ to $p$ is preserved after the reconfiguration. Indeed, we may follow $(x,s)$ then any path towards $p$ in $V(P)$. This yields that $D+X$ is still strongly connected, and concludes the proof for this case.

    In the remainder of the proof, we may assume all arcs in $X_P$ have exactly one endpoint in $P$.
    Take $v \in V(P) \setminus \{s,t\}$ to be any vertex incident to arcs of $X_P$ and let $X_v \subseteq X_P$ be the subset of those arcs.
    Informally, our strategy consists in showing that we can always \say{shift} their $v$ endpoint \say{towards} either $s$ or $t$.
    Consider $v$, and let $u$ (resp. $w$) be the previous (resp. next) vertex incident to $X_P$ when following $P$.
    When this is not defined, we let $u=s$ (resp. $w=t$).
    We define $X_u$ (resp. $X_w$) to be the subset of arcs obtained from $X_v$ by substituting every $v$ endpoint of arcs in $X_v$ with $u$ (resp. with $w$), see~\Cref{fig:pushing-directed}.
    We will show that $X_v$ can be replaced with either $X_u$ or $X_w$ in $X$ while maintaining a solution.
    Observe first that by definition of $u,w$, and since the only endpoint of $X_v$ on $P$ is $v$, both $D + X - X_v + X_u$ and  $D + X - X_v + X_w$ are plane. Therefore, letting $D' = D + X - X_v$, it suffices to show that either $D'+X_u$ or $D'+X_w$ is strongly connected.

    First, observe that $X_v$ forms an oriented star incident to all terminal components of $D'$.
    In $D'$, let $Y'_v$ be the strong component of containing $v$,  $(S'_i)_i$ be the remaining source components, and $(T'_j)_j$ be the remaining sink components.
    Since $D' + X_v$ is strong, $X_v$ contains an arc from $v$ to every $S'_i$, and from every $T'_j$ to $v$.
    In particular, we have the following property:
    \begin{equation}\label{eq:sca}
        \text{$X_v$ induces a directed path from every sink to every source of $D'$}
    \end{equation}
    It is easy to see that any plane completion of $D'$ satisfying~\cref{eq:sca} (in place of $X_v$) is also a solution to $D'$. We show that either $X_u$ or $X_w$ must satisfy~\cref{eq:sca}.
    By definition, both $X_u$ and $X_w$ form stars which induce a directed path of length two from every $T'_j$ to every $S'_i$.
    In particular, if $Y'_v$ is not a terminal component, all terminal components of $D'$ are of the form $(S'_i)_i$ or $(T'_j)_j$, meaning both $X_u$ and $X_w$  satisfy~\cref{eq:sca}. We then update $X$ to $X - X_v + X_u$, which maintains a solution for $D$.
    Now, if $Y'_v$ is a source component of $D'$, the existence of the directed subpath $(s, \ldots ,u \ldots ,v)$ of $P$ implies $u \in Y'_v$ as well. Then, $X_u$ contains arcs from all sinks of $D'$, $(T'_j)_j$, to $Y'_v$, so $X_u$ satisfies~\cref{eq:sca}, so updating $X$ to $X - X_v + X_u$ maintains a solution for $D$.
    See~\Cref{fig:pushing-directed} for an illustration of this case.
    Symmetrically, when $Y'_v$ is a sink component, the subpath $(v, \ldots, w , \ldots t)$ yields that $w \in Y'_v$, and updating $X$ to $X - X_v + X_w$ yields a solution for $D$.

    \begin{figure}[h]
        \centering
        \includegraphics[scale=1]{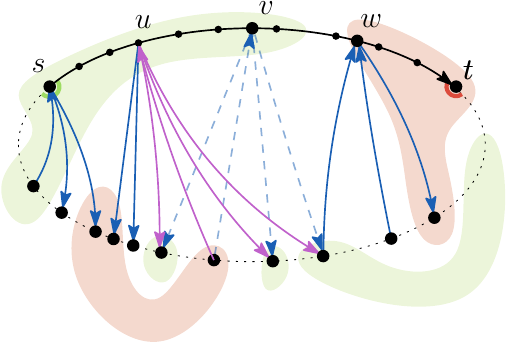}
        \caption{A directed path $P$ from a local source $s$ to a local sink $t$ of a face, with arcs of a solution incident to it shown in blue. Removing the arcs incident to vertex $v$ yields several source and sink components (in green and red), one of which is a source containing both $u$ and $v$. Shifting their endpoints to $u,$ as shown in purple, yields a solution incident to fewer internal vertices of $P.$}\label{fig:pushing-directed}
    \end{figure}

    As long as there exist endpoints of $X_P$ not incident to $s$ or $t$, we reiterate the operation described in the last paragraph.
    This process terminates, because the number of endpoints of $X_P$ on $P$ other than $s$ or $t$ strictly decreases. 
    Eventually, all endpoints of $X_P$ on $P$ (and only those) will have been reconfigured to either $s$ or $t$, while maintaining that $X$ is a solution. This achieves to show the reconfiguration of $X_P$, and therefore concludes the proof.
\end{proof}
 

\subsection{Bounding supported completions in alternating faces}\label{ssec_dir-alternating-faces}

From now on, \Cref{lem_supported-sol-dir-psca} allows us to narrow our search only to supported solutions, which we do by considering their restrictions, \emph{supported completions} within each face.
In this subsection, we bound the number of supported completions of size $k$ across all alternating faces, as a function of $k$.
The first step is to obtain such a bound for a single face, first in terms of the number of its local terminals.
\begin{lemma}\label{lem_guess-alt-dir}
For any face $F$ of $D$, there exist at most $2^{O(\lt(F))}$ supported completions of $F$, and those can be generated within $2^{O(\lt(F))}$-time.   
\end{lemma} 
\begin{proof}
Let $F$ be a face of $D$, and let $\lt=\lt(F)$.
Note that the embedding of any supported completion $X$ of $F$ is fully determined by the local terminals (labelled vertices of $F$) to which it is incident.
Therefore, our goal is to enumerate all possible labelled outerplanar digraphs between the local terminals.
For our purposes, we may consider local terminals as distinct vertices on a cyclic face, up to contracting them afterwards. Each digraph we enumerate can thus be seen as a labelled directed triangulation on $\lt$ vertices.
It is well known that the number of labelled triangulations of a polygon with $\lt$ angles corresponds to the Catalan number $C_{\lt-2},$ which is $2^{O(\lt)}$, see Walkup~\cite{walkupNumberPlaneTrees1972}.
These triangulations can also be enumerated in $2^{O(\lt)}$ time using the algorithm of Bespamyatnikh~\cite{bespamyatnikhEfficientAlgorithmEnumeration2002}.
Then, each triangulation contains $2 \lt -3$ edges, and the number of ways to choose either a non-arc, an arc (in either direction) or a digon is bounded by $4^{2\lt-3} = 2^{O(\lt)}$.
Combining the above, the number of labelled outerplanar digraphs, thus of supported completions of $F$, is bounded by $2^{O(\lt)}$, and their enumeration takes $2^{O(\lt)}$-time.
\end{proof} 

The following bounds the number of local terminals across all alternating faces, thus by~\Cref{lem_guess-alt-dir} the number of their supported completions.
\begin{restatable}{lemma}{lemboundsumlt}\label{lem_bound-lt-alt}
    Every acyclic instance $(D,k)$ of \dirPSCA is either a \no-instance, or satisfies: 
    $$\sum_{F\in \AF(D)} \lt(F)< 8k,$$ 
    where $\AF(D)$ is the set of alternating faces of $D$.
\end{restatable}
\noindent
The proof of~\Cref{lem_bound-lt-alt} is delayed to~\Cref{ssec_dir_bound_local}, but is easily derived from results by Guattery and Miller~\cite{guatteryContractionProcedurePlanar1992}.
Combining \Cref{lem_guess-alt-dir} and \Cref{lem_bound-lt-alt}, we derive the following bound on possible supported completions across alternating faces, and their enumeration in \FPT time.
\begin{corollary}\label{cor_guess-all-alt-dir}
    Every acyclic instance $(D,k)$ of \dirPSCA is either a \no-instance, or has $2^{O(k)}$ different supported completions within alternating faces of $D.$
    Furthermore, this can be decided and these completions can be generated within $2^{O(k)}$-time.
\end{corollary}

\subsection{Computing the rest of the solution in simple faces}\label{ssec_dir_FPT}

With \Cref{cor_guess-all-alt-dir} in hand, we may branch over the $2^{O(k)}$ possible supported completions in the alternating faces.
Then, if a supported solution $X$ to $(D,k)$ exists, we may assume we are in a branch that has correctly computed its restriction to the set of alternating faces $X_\AF =\cup_{F\in \AF} X_F$.
Since we know $X$ is a solution for $D$, $X\setminus X_{\AF}$ is a supported solution of $(D+X_{\AF},k-|X_{\AF}|)$ where all arcs are completed in simple faces.
Therefore, instead of recovering $X\setminus X_{\AF}$ exactly, we may look for \textsl{any} supported solution to $(D+X_{\AF},k-|X_{\AF}|)$ within (only) simple faces of $D$.

Before computing such a solution, the crucial observation is that any simple face $F$ admits exactly one non-empty supported completion, which is the arc from its local sink $t$ to its local source $s$. Note indeed that this arc always renders $V(F)$ strongly connected, which is the best outcome for a completion within $F$.
Then, to solve $(D+X_{\AF},k-|X_{\AF}|)$ in the simple faces, it suffices to compute the minimal number of arcs from local sinks to (corresponding) local sources needed to achieve strong connectivity.
The most straightforward way to solve this is by a reduction to \tsc{Weighted-SCA}. In this problem, we are given a digraph $D,$ a weight function $w : V(D)^2 \rightarrow \mathbb{R}^+$ for all the possible new arcs, and values $k,\alpha.$ Then, an instance is positive if there exists a set $X\subseteq V(D)^2$ of at most $k$ arcs and total weight at most $\alpha$ such that $D+X$ is strong.

Now, we let the input digraph to \tsc{Weighted-SCA} be $D+X_{\AF}$, and let the weight function equal $k+1$ for all possible arcs, except for those $(t,s)$ where $t$ is the local sink and $s$ is the local source of some simple face $F,$ in which case the weight is $1.$ 
Then, setting both the maximum size and the total weight of a solution to $k-|X_{\AF}|$, one obtains an instance of \tsc{Weighted-SCA} that is clearly equivalent to our problem.
This instance can be solved in $2^{O(k \log k)}n^{O(1)}$, using the algorithm of Klinkby, Misra and Saurabh~\cite{klinkbyStrongConnectivityAugmentation2021}. In turn, this gives a $2^{O(k \log k)}n^{O(1)}$ \FPT algorithm for \dirPSCA.
Let us also note that the computation within simple faces reduces to the (polynomial) \tsc{Minimum Dijoin} problem, by simply adding the arc $(s,t)$ in each simple face, which yields a $2^{O(k)} n^{O(1)}$ algorithm.

\subsection{Adapting the strategy to \spsca}
\label{ssec_adapt}

We highlight here the main obstacles in adapting the algorithm of \dirPSCA to \spsca.

\paragraph{Local terminals in non-acyclic instances.}

The first hurdle one runs into when trying to generalize the approach used for~\dirPSCA is that instances of \sopsca are not equivalent under condensation.
Indeed contracting arcs may increase the solution size, or even prevent augmentability, see~\Cref{fig:noDAG}.
As a result, the local terminals we define for \spsca will have size possibly linear in $n$ (see~\Cref{ssec_structure_instance}), complicating the definition of supported solutions.

\begin{figure}[h]
    \centering
    \includegraphics{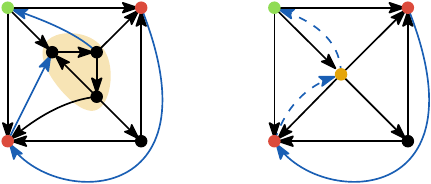}
    \caption{An instance of \spsca, with a non-trivial strong component in orange, and the oriented graph obtained by contracting it into a single vertex.
    The initial instance has a solution of size three, shown in blue, that is not preserved for the condensation, which is a negative instance.}
    \label{fig:noDAG}
\end{figure}

\paragraph{Supported solutions avoiding digons.}

In \spsca, the main difficulty lies in defining a notion of supported completions of a face that avoids digons, but is still constrained enough to allow for bounds in terms of $k$.
The obvious obstacle is that an arc of some solution $X$ cannot be reconfigured onto a pair of vertices that already form an arc in the face at hand.
More challenging is the fact that, when $D$ is not $3$-connected, there may exist arcs between vertices of $F$ that are embedded outside of $F$, but still obstruct the \say{shifting} of arcs in $X$ during reconfiguration.
While such shifts always end up in local terminals in \dirPSCA, this is not guaranteed to happen for instances of \spsca.
Our goal is still to mimic this intuition, by considering \say{leftmost} or \say{rightmost} shifts of $X$ toward local terminals of $F$, while avoiding digons (see~\Cref{ssec_shifts}).
We show that these shifts can only land on bounded number of \say{supports} (sets of angles), but at the cost of losing strong connectivity (see \Cref{ssec_stacks}).
Then, we show how to recover strong connectivity while staying within supports, 
allowing us to define supported solutions.
In turn, this restricts our search to a bounded number of supported completions in each face $F$ (see~\Cref{ssec_supported_bound}).

\paragraph{Branching in alternating faces.}

Although the size of each local terminal can now be very large, their number across alternating faces can still be bounded by a function of $k$ using the arguments used for \dirPSCA (see~\Cref{ssec_structure_instance}).
Then, thanks to our definition of supported completions, we may branch over a bounded number of possible completions across all alternating faces.

\paragraph{Guessing in simple faces.}
Having branched on a possible (supported) completions within alternating faces, we are left with the task of finding a minimum supported solution across all simple faces.
In \dirPSCA, this question was amenable to \tsc{Weighted SCA} thanks to each simple face admitting an unique supported completion: the arc from its local sink to its local source, which rendered $D[V(F)]$ strong. 
This is not achievable in~\spsca, where existing arcs from the local source to the local sink in $D$ may prevent the addition of an arc in the other direction.
In fact, a supported completion of a simple face may need multiple arcs, and may not even render $D[V(F)]$ strong by itself.
In particular, there can be multiple supported completions of $F$ that are pairwise incomparable in terms of connectivity.
The crucial step is to show that this number can still be bounded by a constant (see~\Cref{ssec_supported_bound}), but this prevents a direct reduction to \tsc{Weighted SCA} (or \tsc{Minimum Dijoin}). 
Indeed, we would need to choose one completion for each face in advance, to prevent the algorithm from producing solutions with crossings.
To overcome this, we guess a single allowed completion in each simple face uniformly at random, then show a reduction to the~\textsc{Minimum Dijoin} problem.
When the \say{correct} restriction of a supported solution has been guessed across the \emph{at most $k$} simple faces containing an arc of the solution, our reduction is guaranteed to find some minimum solution.
Since this probability is only a function of $k$, our algorithm can be derandomized in \FPT time using universal sets.

\section{FPT algorithm for \spsca}\label{sec_psca}

This section is dedicated to the proof of~\Cref{thm_fpt}, establishing the fixed-parameter tractability of \psca.
Recall that the technicalities of the current section are motivated by the obstacles discussed in~\Cref{ssec_adapt}. 
We begin with definitions specific to \spsca in~\Cref{ssec_definitions}, notably expressing completions relative to angles instead of labelled vertices.
In~\Cref{ssec_structure_instance}, we describe the structure of our instances: local terminals, intervals (subwalks) between them, and face types.
Our reconfigurations will operate on multicompletions, a (slight) relaxation of completions defined in~\Cref{ssec_multicompletions}.
In~\Cref{ssec_shifts}, we define the shifting operation, acting on endpoints of multicompletions, which will be our main reconfiguration tool.
Then, in~\Cref{ssec_stacks}, we introduce left and right stacks, corresponding to \say{maximally shifted} endpoints of a multicompletion on an interval, and show that these must lie on a bounded number of vertices, called supports.
While a solution cannot always be shifted as such, these operations allow us to define supported solutions in~\Cref{ssec_supported_existence}.
Then, we give bounds on the number of supported completions, in particular those across all alternating faces in~\Cref{ssec_supported_bound}.
We show the (derandomized) algorithm computing a minimum solution for simple faces in~\Cref{ssec_simple}, which we combine in~\Cref{ssec_fpt_psca} with a branching on alternating faces to yield our \FPT algorithm for \spsca.

\subsection{Definitions}\label{ssec_definitions}


\paragraph{Angles and boundaries.}
Given a face $F$, an \emph{angle} $w = \widehat{e v e'}$ of $F$ is an ordered triplet given by a vertex $v$ and its two incident arcs $e,e'$ on $F$.
When referring to $w$ as a vertex, we always mean $v$.  
Unless specified otherwise, each $w$ is ordered clockwise following $F$, and $e$ is the arc \emph{preceding} $v$ while $e'$ is the one \emph{following} $v$.
We then let $W(F)$ be the set of angles of $F$, and $W(D)$ be the set of angles across all faces of $D$.
We identify each face $F$ by its \emph{boundary}, which is now the circular list of angles $(w_1,w_2,...,w_r)$ following the closed walk of $F$ clockwise.
For simplicity of notation, we regularly identify the \textsl{labelled} vertex $v_i$ of the boundary with its \emph{corresponding angle} $w_i = \widehat{a_i v_i a_{i+1}}$.
Then, the \emph{interval} of $F$ from $w_i$ to $w_j$ (or $v_i$ to $v_j$), denoted by $I=[w_i,w_j]$, is the (clockwise) subwalk of $F$ from $w_i$ to $w_j$.
Within our proofs, all intervals are considered as clockwise, or \say{left} to \say{right}, but the results hold for intervals ordered anti-clockwise by symmetry.

\paragraph{Completions between angles and their endpoints.}
In this section, a \emph{completion} $X$ of $D$ is specified as an $X \subseteq W(D)^2$ such that $D+X$ is plane and oriented, where $e=(w,w') \in X$ is embedded between the two \textsl{angles} $w$ and $w'$ of $D$.
Observe that considering arcs of $X$ as such uniquely determines their embedding in $D+X$.
Given an element $e=(w,w')\in X$, $w^e$ denotes the \emph{endpoint} of $e$ incident at angle $w$, and we let $\Xend$ be the set of endpoints of $X$.
For any set of endpoints $U$,  $W(U)$ is the set \emph{angles}, and $V(U)$ is the set of vertices, corresponding to $U$. For a completion $X$, we let $W(X) = W(\Xend)$ and $V(X) = V(\Xend)$.

Given an interval $I$ of $F$, $\Xend_I$ denotes the set $\Xend \cap I$ of endpoints of $X$ on $I$. Recall $X_I$ is the set of arcs of $X$ with \textsl{at least one} endpoint on $I$. 
In particular, when $I$ is the whole boundary of $F$, $X_F$ (resp. $\Xend_F$) consists of all arcs of $X$ (resp. all endpoints of $\Xend$) embedded in $F$.
When $I=[w]$ consists of a single vertex, we may simply denote this set by $X_w$, which we refer to as the endpoints of $X$ \emph{at} $w$.

\subsection{Structure of the instances}\label{ssec_structure_instance}

Throughout the section, we let $(D,k)$ be an instance to \spsca, where $D$ is a plane oriented graph, and ask for a solution $X$ of size at most $k$ such that $D+X$ is still plane and oriented.
In this subsection, we define the intervals in which our reconfiguration happens: local terminals and interval dipaths, then distinguish face types according to local terminals.

To deal with non-acyclic instances, we should capture the \say{equivalence} of vertices in the same strong component with respects to strong connectivity augmentations.
This is done through the following.
\begin{definition}[Strong intervals]
    A {\em strong interval} of a face $F$ is a maximal interval of its boundary such that all corresponding vertices belong to the same strong component of $D$.
\end{definition}

\subsubsection{Local terminals}

We now define local terminals, see~\Cref{fig:local-ter} for an illustration.
\begin{definition}[Local terminals]\label{def_local-terminals}
Given a plane digraph $D$ (allowing loops) and a face $F$ of $D$, a \emph{local terminal} of $F$ is a strong interval $I = [w_i,w_j]$ such that either:
\begin{itemize}
    \item  The arcs preceding $w_i$ and following $w_j$ along $F$ are both out-arcs, then $I$ is a \emph{local source}.
    \item  The arcs preceding $w_i$ and following $w_j$ along $F$ are both in-arcs, then $I$ is a \emph{local sink}.
\end{itemize}
\end{definition}
\noindent
Local terminals can also be seen as the strong intervals of $F$ forming dicuts in the digraph induced by the arcs of $F$.
The need for loops in~\Cref{def_local-terminals} is merely a technical artifact used in~\Cref{lem_bound-lt-alt-oriented}, but none of our instances contain loops.

\begin{figure}
    \centering
    \includegraphics[scale=1.3]{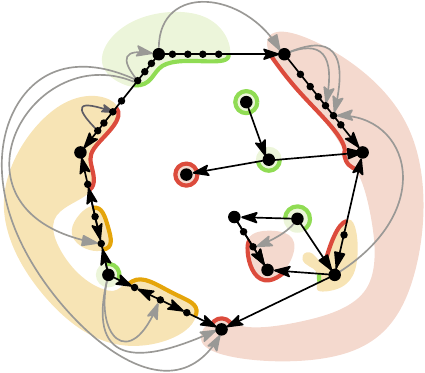}
    \caption{A face $F$ of $D$, with arcs of $D[V(F)]$ embedded outside $F$ shown in gray. The strong components of $D$ are shown as clouds in green (source), red (sink) and orange (non-trivial intermediate). Each consecutive intersection of these components with the boundary of $F$ forms a strong interval. Strong intervals are shown by thick lines, in green for local sources, red for local sinks, and orange otherwise.}
    \label{fig:local-ter}
\end{figure}

Observe that a (possibly trivial) terminal component of $D$ necessarily results in at least one local terminal in each face it intersects. Still, local terminals may also stem from non-terminal components, and a strong component can lead to several local terminals, even in the same face.
Now, any strong interval which is not a local terminal must be contiguous to exactly one in-arc and one out-arc on the boundary. This is in particular the case for strong intervals between two consecutive local terminals of $F$.
In turn, consecutive local terminals are of different types, meaning local sources and local sinks alternate following $F$. Therefore, they are equal in number, and $\lt(F)$ is always even.

\subsubsection{Interval dipaths between local terminals}

We now define interval dipaths, which are the intervals between local terminals, then show that these behave essentially like directed paths, connectivity-wise.
\begin{definition}
    Given a face $F$ of $D$, an {\em interval dipath} of $D$ is a maximal interval $P$ of the boundary that does not intersect a local terminal.
\end{definition}
\noindent
Observe that any interval dipath $P=[w_1,w_r]$ must be contiguous to a local source $S$ and a local sink $T$, and we say that $P$ is from $S$ to $T$.
Up to symmetry, we will always assume that, clockwise around $F$, $w_1$ is the angle following some $w_0 \in S$, and $w_r$ is an in-neighbour of a following $w_{r+1} \in T$. 
We may therefore use the terms {\em left} to refer to lower indices (towards $S$) and {\em right} referring to higher indices (towards $T$).

\begin{lemma}\label{lem_paths-between-localT}
Let $P=[w_1,w_r]$ be an interval dipath from a local source $S$ to local sink $T$.
\begin{enumerate}
    \item For any vertices $s\in S$, $t\in T$, and any vertex $w_i\in P$, there exists a directed path from $s$ to $w_i$ and one from $w_i$ to $t$.
    \item For any $i,j \in [1,r]$ with $i<j$, there exists a directed path from $w_i$ to $w_j$.
\end{enumerate}
\end{lemma}
\begin{proof}
    We show the first path, from a vertex $s \in S$, obtaining the path towards $t$ can be done symmetrically.
    Let $w_0$ be the angle outside $P$ preceding $w_1$ in $F.$
    As $s$ and $w_0$ belong to the same strong component of $D$, there is a dipath from $s$ to $w_0$, and the arc $(w_0,w_1)$ yields the dipath towards $w_1$. To show the paths towards all angles of $P$, it suffices to use the second property, which we prove in the following. 

    For the second property, consider two angles $w_i$ and $w_j$ with $i<j$. Let $I_0,\ldots,I_q$ be the strong intervals of $P$ ordered from left to right, such that $w_i\in I_0$ and $w_j\in I_q$ (possibly with $q=0$). As these are not local terminals, the last angle of $I_{\ell}$ has an out-arc to the first one of $I_{\ell+1}$, for $i \in [0,q-1]$.
    These arcs, along with dipaths in each strong interval between their endpoints, allow us to build the desired dipath from $w_i$ to $w_j$.
\end{proof}

\subsubsection{Simple and alternating faces}

We are now ready to define simple and alternating faces. 
Let $\lt(F)$ denote the (even) number of local terminals around a face $F$, recall that $\lt(F)$ is necessarily even.
As any face $F$ intersects at least one strong component of $D$, $F$ contains at least one (possibly trivial) strong interval $I=[a,b]$.
If $I$ consists in the whole boundary of $F$, $V(F)$ belongs to the same strong component, meaning $I$ cannot be a local terminal, and $\lt(F)=0$, and we say that $F$ is \emph{strong}. Observe already that it is unnecessary to complete a strong face with arcs in a minimal solution, which is why these can mostly be ignored in the following.
A face $F$ is \emph{simple} if it contains two local terminals, and $\SF(D)$ denotes the set of simple faces of $D$.
A face $F$ is \emph{alternating} if it contains at least four local terminals, and $\AF(D)$ denotes the set of alternating faces of $D$.

\subsubsection{Bounding local terminals in alternating faces}
\label{ssec_nb-local-ter-oriented}

The starting observation here is that, when the budget $k$ is bounded, the number of terminal components \textsl{of $D$} is bounded as well. Indeed, for any $Y$ belonging to the set of terminal components $\cT(D)$, a solution $X$ must contain an arc with an endpoint in $Y$, yielding the following.
\begin{observation}\label{obs_bound_terminals_k}
    For any positive instance $(D,k)$, we have $|\cT(D)| \leq 2k$
\end{observation}

The next lemma bounds the number of local terminals across \textsl{all} alternating faces in terms of $|\cT(D)|$. With the observation above, this yields in particular a bound in terms of $k$. For the sake of the induction, we show this more generally for plane digraphs allowing loops.
\begin{lemma}\label{lem_bound-lt-alt-oriented}
For every plane digraph $D$ allowing loops, we have $$\sum_{F\in \AF(D)} \lt(F) \leq 4 |\cT(D)| - 8$$
\end{lemma}
\begin{proof}
    Successively contracting any (non-loop) arc lying within a strong component of $D,$ yields a plane digraph $D'$ that is a DAG with loops. 
    More precisely, deleting the loops of this digraph yields an oriented DAG. Recall that the definitions of local terminals and alternating faces also hold for DAGs with loops. 
    
    \begin{claim}
    Given a digraph $D_1$, and a digraph $D_2$ obtained from $D_1$ through the contraction of a single non-loop arc $uv$ lying inside a strong component of $D_1,$ we have 
    $$|\cT(D_1)| = |\cT(D_2)| \text{ and } \sum_{F\in \AF(D_1)} \lt(F) = \sum_{F\in \AF(D_2)} \lt(F).$$
    \end{claim}
    \begin{proofof}{Claim}
        Note that the strong components of $D_1$ are bijectively mapped to the strong components of $D_2,$ in such a way that the terminal components of $D_1$ are mapped to the terminal components of $D_2.$
        Also, the faces of $D_1$ are bijectively mapped to the faces of $D_2.$ These bijections are such that if a face $F_1$ of $D_1$ is mapped to a face $F_2$ of $D_2,$ we have that $\lt(F_1)=\lt(F_2).$ This implies that the faces of $\AF(D_1)$ are mapped to those of $\AF(D_2).$ Hence, the claim holds.
    \end{proofof}

    It is thus sufficient to prove the lemma for the plane oriented DAGs with loops. Let us proceed by induction on the number of these loops. The initial case, with a loopless $D,$ follows by \Cref{lem_bound-lt-alt} and  we distinguish two cases for the induction.
    First, consider the case where there exists a face $F$ with only one arc on its boundary, that is, a loop $e = (u,u)$. Delete $e$, and observe that for the other face incident to $e,$ the number of local terminals is unchanged, such as for all the other faces. Hence, by induction the inequality of the lemma holds.

    Otherwise, we may consider the case where a loop $e=(u,u)$ separates two non-empty regions. Let $D_1$ and $D_2,$ be the subdigraphs of $D$ induced by the arcs inside or outside $\ell,$ respectively. One can observe that by definition, $\ell$ is contained in a strong interval for any of its incident faces. Hence, for any face $F$ of $D,$ the corresponding face of $D_1$ or $D_2,$ according to the position of $F$ with respect to $\ell,$ keeps the same number of local terminals. Hence, 
    $$ \sum_{F\in \AF(D)} \lt(F) = \sum_{F\in \AF(D_1)} \lt(F) + \sum_{F\in \AF(D_2)} \lt(F).$$
    
    Furthermore, if the strong component $C$ of $v$ in $D$ (which has only $v$ as vertex) is terminal in $D,$ then in each of $D_1$ and $D_2$ the strong components of $v,$ $C_1$ and $C_2$ respectively, are also terminal in their respective digraphs, while if $C$ is not terminal at least one of $C_1$ or $C_2$ is not terminal. Hence,  $ |\cT(D_1)| + |\cT(D_2)| \le |\cT(D)| +1,$ which implies that: 
    $$ \sum_{F\in \AF(D)} \lt(F) \le 4|\cT(D_1)| -8 + 4|\cT(D_2)| - 8 \le 4|\cT(D)| -8.$$
\end{proof}

\subsection{Multicompletions}\label{ssec_multicompletions}

For the sake of reconfiguring a completion $X$, we relax the constraint of $X$ having to induce an oriented graph, as long as it does not create a digon with $D$.
\begin{definition}[Multicompletion]\label{def_multicompletion}
    Given a plane oriented graph $D$, a \emph{multicompletion} of $D$ is a multiset $X$ on $W(D)^2$ such that: $D+X$ is a plane multidigraph; and any digons, parallel arcs, or loops of $D+X$ belong entirely to $X$.
\end{definition}
\noindent
As is the case for completions, the embedding of an arc in a multicompletion is fully determined by its angles (up to permuting parallel arcs and loops around a fixed angle).

\subsubsection{Minimally strong multicompletions are solutions}

Because we consider multicompletions, the reconfiguration of solution $X$ may introduce digons, parallel arcs, or loops into $X$.
The following lemma allows us to ignore this issue for most applications: when shifts preserve strong connectivity, they result in a (proper) solution.   
\begin{lemma}\label{lem_no_digon_multicompletion}
    Given an oriented graph $D$ whose underlying graph is connected, any minimal multicompletion $X$ such that $D+X$ is strongly connected satisfies:
    \begin{itemize}
        \item No arc of $X$ lies within a single strong component of $D$,
        \item At most one arc of $X$ lies between any pair of strong components.
    \end{itemize}
    In particular, $X$ is a solution.
\end{lemma}
\begin{proof}
    It is clear that any arc of $X$ between two vertices of the same strong component of $D$ may be removed from $D+X$.
    When there are two arcs $e,e' \in X$ from a component $U$ to a component $U'$, either of them can be removed.
    In both cases, the vertices connected using the removed arc can always be replaced by a directed path in the resulting digraph.

    Now, assume that for strong components $U,U'$ of $D$, there exist arcs $e,e' \in X$ with $e$ directed from $U$ to $U'$, and $e'$ from $U'$ to $U$. 
    We argue that either $e$ or $e'$ can be removed from $X$ while maintaining a solution.
    Let $D' = D + X - \{e,e'\}$, and note that all terminal components of $D'$ must contain either $U$ or $U'$ (since they are not terminal in $D$).
    Let us call $Y,Y'$ the components containing $U,U'$ respectively.
    Because the underlying graph of $D$ is connected, $Y$ and $Y'$ cannot be isolated (that is, both source and sink), meaning w.l.o.g. we can assume $Y$ is the unique source and $Y'$ is the unique sink.
    In this setting, there exists a dipath from (any vertex of) $Y$ to (any vertex of) $Y'$.
    Then, in any dipath of $D+X$ using $e$, $e$ may be replaced by a disjoint dipath, meaning $D'+e'$ is strongly connected. That is, $X-e$ is a multicompletion achieving strong connectivity, which contradicts minimality.

    In particular, all the above forbids any loops, digons or parallel arcs fully in $X$, meaning none exist in $D+X$ by \Cref{def_multicompletion}, so $X$ is also a solution.
\end{proof}

\subsubsection{Backward arcs on interval dipaths}

For an interval $I$, the arcs of $X_I$ having both endpoints on $I$ are called \emph{backward arcs of $I$}.
By \Cref{lem_no_digon_multicompletion}, a minimum solution cannot contain a backward arc within a local terminal, so we will only be concerned with backward arcs within interval dipaths.
The following observations, derived from planarity and from~\Cref{lem_paths-between-localT}, allow us to further constrain occurences of backward arcs on interval dipaths.
\begin{observation}\label{obs_backward}
If $X$ is a minimal solution, for any backward arc $e=(w_j,w_i)$ of an interval dipath $P$:
\begin{enumerate}[label=(\roman*)]
    \item no arcs of $X$ have both endpoints in $[w_{i},w_{j}]$, that is, backward arcs are not nested;
    \item $i<j$, that is, $e$ is directed from right to left.
\end{enumerate}
\end{observation}

\subsection{Shifting endpoints of a completion}\label{ssec_shifts}

In this subsection, we define shifts, the operation we use throughout our reconfiguration, then show how to use shifts to gather endpoints onto a common angle.
\begin{definition}[Shifts]\label{def_shift}
    Let $I = [w_1,w_r]$ be an interval of some face $F$, $X$ be a multicompletion, and $w_i^e \in \Xend_I$.
    Let $w_m^a \neq w_i^e$ be the rightmost endpoint of $\Xend_I$ with $m \leq i$, letting $w_m=w_1$ if undefined.
    Let $w_M^b \neq w_i^e$ be the leftmost endpoint of $\Xend_I$ with $M \geq i$, letting $w_M=w_r$ if undefined.

    We define the following shifts, substituting $w_i^e$ for $w_j^e$ in $e$, when $w_j$ is defined: 
    \begin{itemize}
        \item When $m < i$, the \emph{left shift} of $w_i^e$ is the substitution of $w_i^e$ for $w_j^e$, minimizing $j \in [m,i[$ and maintaining a multicompletion.
        \item When $M > i$, the \emph{right shift} of $w_i^e$ is the substitution of $w_i^e$ for $w_j^e$ maximizing $j \in ]i,M]$ and maintaining a multicompletion.
    \end{itemize}
\end{definition}

\begin{definition}[Shifts]\label{def_shift}
    Let $I = [w_1,w_r]$ be an interval of some face $F$, $X$ be a multicompletion, and let $e\in X_F$ be an arc with an endpoint $w_i^e$ in $I$.
    Let $w_m^a$ and $w_M^b$ be the endpoints of $\Xend_I$ preceding and following $w_i^e$, with $w_m=w_1$ and $w_M=w_r$ if undefined.
    \begin{itemize}
        \item The \emph{left shift} of $w_i^e$ is the replacement of this endpoint with the leftmost angle $w_j$ with $j \in [m,i[$ maintaining a multicompletion, that is, maintaining that $D+e$ is plane and oriented.
        \item The \emph{right shift} of $w_i^e$ is the replacement of this endpoint with the rightmost possible angle $w_j$ with $j \in ]i,M]$ maintaining a multicompletion.
    \end{itemize}
\end{definition}
Observe that, for $e=(u,w_i)$, the left shift of $w_i^e$ may indeed create a parallel arc or a digon in $X$, if $X$ already contains $(u,w_j)$ or $(w_j,u)$, or even a loop when $u = w_m$.

\subsubsection{Gathering endpoints}

The following operation is based on shifts, and will allow us to gather endpoints of a solution that are spread out into a single angle, while maintaining a multicompletion.
\begin{lemma}\label{lem_gather}
    Consider any interval $I = [w_1,w_r]$ of a face $F$, and any multicompletion $X$ incident to two angles $w_i,w_j$ ($i <j$), such that $X$ is not incident to any $w_{\ell}$ with $\ell \in ]i,j[$.
    Then, either all the endpoints at $w_i$ can be successively shifted right onto $w_j$, or all the endpoints at $w_j$ can be successively shifted left onto $w_i$.
\end{lemma}
\begin{proof}
    The following considers all arcs irrespective of their orientation, and we use $uv$ to mean both $(u,v)$ and $(v,u)$.
    Recall that by \Cref{def_shift}, any right (resp. left) shift of an endpoint in $\Xend_{w_i}$ (resp. $\Xend_{w_j}$) results in an endpoint on $[w_i,w_j]$. 
    Consider the successive right shifts of endpoints $w_i^e \in \Xend_{w_i}$. 
    If at any point, the right shift of some $w_i^e$ would not result in $w_j$, we stop the process. Otherwise we obtain the desired result. 

    In the case where we stopped the process, say at some $e = u w_i$, there must be an arc not internal to $F$ between (vertices corresponding to) $u$ and $w_j$ in $D$.
    Therefore, at this point, any $e' = u' w_j \in X_{w_j}$ is such that $u' \in ]w_j,u[$, else $u' w_j$ would cross $u w_i$.
    Therefore, no such $u'$ can be a neighbour of $w_i$, for otherwise $u'w_i$ would cross $u w_j$ in $D$.
    Now, we consider symmetrically the successive left shifts of $\Xend_{w_j}$.
    Then, since $w_i$ is not a neighbour of any non-$w_j$ endpoint, all such shifts result in $w_i$, concluding the proof.
\end{proof}

\subsection{Stacks of a solution and their supports}\label{ssec_stacks}

We now define the left and right stacks of a multicompletion on an interval, which are \say{maximal} subsets of endpoints that cannot be shifted. Then, we will show that these stacks can only correspond to a bounded number of subsets of angles called supports.
\begin{definition}[Stacks]\label{def_stack}
    Let $I$ be an interval of some face $F$, $X$ be any multicompletion, and endpoints $\Xend_I$ ordered from left to right according to the incidence of their corresponding arcs on $I$.
    \begin{itemize}
        \item The left stack of $X$ on $I$, denoted $\Xend_I^L$, is the maximal prefix of $\Xend_I$ containing no endpoint that can be shifted left. 
        \item The right stack of $X$ on $I$, denoted $\Xend_I^R$, is the maximal suffix of $\Xend_I$ containing no endpoint that can be shifted right.
    \end{itemize}
\end{definition}
\noindent
Note that these sets may overlap, which will not be an issue as our goal is only to reconfigure endpoints to be in their union. Observe also that $W(\Xend_I^L)$ must be disjoint from $W(\Xend_I \setminus \Xend_I^L)$, by the definition of shifts.

In order to constrain the incidence of stacks on intervals, we will need the following observation, stemming from the fact that each face induces an outerplanar digraph.
\begin{observation}\label{obs_unique-common-neighbour}
    For any face $F = (w_1,...,w_r)$, any two consecutive $w_{i},w_{i+1}$ on the boundary have at most one common neighbour in $V(F)$.
\end{observation}

We now define the set of left supports $Z^L_I(q)$, and the set of right supports $Z^R_I(q)$. 
\begin{definition}[Supports]\label{def_supports}
    Consider any face $F$ of $D$, and any interval $I = [w_1,w_r]$ of $F$, we define a set $Z^L_I(q)$, consisting of subsets of $I$ of size $q$ called \emph{left $q$-supports} of $I$.

    We first let $Z^L_I(1)$ contain $\{w_1\}$, $\{w_2\}$.
    Then, if $w_1,w_2$ have a common neighbour $u$ on $F$, add $\{w_{h}\}$ to $Z^L_I(1)$, where $w_h$ is the leftmost non-neighbour of $u$, if it exists.

    For $q>1$, we define $Z^L_I(q)$ from $Z^L_I(q-1)$ as follows.
    Iterating over all $B \in  Z^L_I(q-1)$, let $w_i$ be the element with the highest index in $B$, and add the following left $q$-supports to $Z^L_I(q)$:
    \begin{itemize}
        \item If $i<r$, add $B \cup w_{i+1}$.
        \item If there is a common neighbour $u$ to $w_i,w_{i+1}$, add $B \cup w_h$, where $w_h$ is the leftmost non-neighbour of $u$, if it exists.
    \end{itemize}

    The set of right $q$-supports of $I$, denoted $Z^R_I(q)$, is defined symmetrically by replacing \say{left} with \say{right}, and following $I$ in reverse from $w_r$ to $w_1$.
\end{definition}

\begin{observation}\label{obs_support-bounded-size-time}
  $Z^L_I(q)$ and $Z^R_I(q)$ have size at most $3\cdot 2^{q-1} = 2^{O(q)}$, and can be constructed in time $2^{O(q)} n^{O(1)}$.
\end{observation}

We now show that the \textsl{angles} of a left (resp. right) stack coincide with some left (resp. right) support, that is, with some element of $Z^L_I(q)$ ($Z^R_I(q)$).
\begin{lemma}\label{lem_stacked-support-bound}
    Consider an interval $I$ and a multicompletion $X$.
    Let $\Xend_I^L,\Xend_I^R$ be the left and right stacks of $X$ on $I$, and $q_L = |W(\Xend_I^L)|, q_R = |W(\Xend_I^R)|$ be their number of angles, then:
    \begin{itemize}
        \item $W(\Xend_I^L) \in Z^L_I(q_L)$.
        \item $W(\Xend_I^R) \in Z^R_I(q_R)$.
    \end{itemize}
\end{lemma}
\begin{proof}
    We proceed by induction on $q$, showing that, for any prefix $\Omega$ of $\Xend_I$, such that $|W(\Omega)| = q$ and no endpoint of $\Omega$ can be shifted left, $W(\Omega) \in Z^L_I(q)$. 
    This implies the result for left stacks by \Cref{def_stack}, and the proof translates to right stacks by symmetry.
    Throughout the proof, we refer to arcs $(u,v)$ and $(v,u)$ as their underlying edge $uv$, and argue on the latter with no distinction on the orientation.

    If $q=1$, we have $W(\Omega) = \{w_j\}$.
    If $j=1$ or $j=2$, then indeed $\{w_j\} \in Z_W^L(1)$ by \Cref{def_supports}.
    Otherwise, $j \geq 3$, and we let $u w_j$ be the leftmost arc of $X_{w_j}$ at $w_j$.
    Since $u w_j$ cannot be shifted left, $u$ must be a common neighbour of $w_1,w_2$ in $D$, which is unique by \Cref{obs_unique-common-neighbour}. Moreover, $w_j$ is the leftmost non-neighbour of $u$ on $I$, yielding $\{ w_j \} \in Z_W^L(1)$.

    Assume the induction holds up to $q-1$, and let $\Omega$ be a prefix of $\Xend_I$ with $|W(\Xend_I)|=q$, such that no endpoint can be shifted left.
    Let $\Omega'$ be the maximal prefix of $\Omega$ on $q-1$ angles, which belongs to $Z_I^L(q-1)$ by induction.
    Let $w_i$ be the rightmost angle of $W(\Omega')$, and $w_j$ be the angle of $W(\Omega) \setminus W(\Omega')$, so $j>i$.
    If $j = i+1$ we know the angles of $\Omega$ correspond to a support of $Z_I^L(q)$ by \Cref{def_supports}.
    Otherwise, $j \geq i+2$, and we let $w_j^e \in \Omega$ be the leftmost endpoint of at $w_j$, say with $e=u w_j$.
    If $w_j^e$ cannot be shifted left, $u$ is in particular the unique common neighbour of $w_i$ and $w_{i+1}$ in $D[V(F)]$, and therefore $w_j$ is the leftmost non-neighbour of $u$, again ensuring $W(\Omega) \in Z_I^L(q)$. This concludes the induction, and shows the result since $\Xend_I^L$ is a prefix of endpoints that cannot be shifted left.
\end{proof}

\subsection{Existence of supported solutions}\label{ssec_supported_existence}

We begin by defining supported completions and solutions, then show every solution can be reconfigured through shifts into one that is supported.
\begin{definition}[Supported completions]\label{def_supported}
    Given a completion $X$, an interval $I$, with $|W(\Xend_I)|=q$, we say $X$ (or $\Xend_I$) is \emph{supported on $I$} if there exist $B^L \in Z_I^L(q),B^R\in Z_I^R(q)$ such that $W(\Xend_I) \subseteq B^L \cup B^R$. We then call $B^L \cup B^R$ a \emph{support} of $X$ on $I$.
    
    For any face $F$, $X$ (or $X_F$) is \emph{supported on $F$} if $X$ is supported on every local terminal and interval dipath of $F$.
    Then, a completion $X$ is \emph{supported} if it is supported in every face of $D$.
\end{definition}

We will routinely use the following trivial condition for a shift to maintain a solution.
\begin{observation}\label{obs_shifting-safe}
    Let $X$ be a solution, and $w^e \in \Xend_I$ for some interval $I$.
    If a shift of $w^e$ results on an angle in the same strong component of $D+X-e$ as $w^e$, then $X$ is still a solution.
\end{observation}

\subsubsection{Local terminals}

We now show that endpoints of a solution on an interval dipath can always be shifted onto a support.
\begin{lemma}\label{lem_supported-local-terminals}
    For any local terminal $I$, and any solution $X$, there exists a sequence of shifts of $\Xend_I$ resulting in a solution that is supported on $I$.
\end{lemma}
\begin{proof}
    Let $X$ be such a solution, and $\Xend_I$ be its endpoints on $I = [w_1,w_r]$.
    As long as $\Xend_I$ contains an endpoint $w_i^e$ that can be shifted left on $I$, we perform that shift.
    At any step, $w_i^e$ is shifted onto some $w_m \in I$, which maintains a solution thanks to \Cref{obs_shifting-safe}. Indeed, $w_m$ belongs to the same strongly connected component as $w_i$ in $D$, thus also in $D+X-e$.
    The process terminates, because the sum of the indices of endpoints $\Xend_I$ on $I$ strictly decreases at each step.
    Eventually, we obtain a solution $X$ where no endpoint of $\Xend_I$ can be shifted left on $I$, meaning $\Xend_I$ coincides with its left stack $\Xend_I^L$. Then, \Cref{lem_stacked-support-bound} yields that $W(\Xend_I^L)$ belongs to $Z_I^L(|W(X_I)|)$.
\end{proof}

\subsubsection{Interval dipaths}

Unlike for local terminals, shifting the endpoints of a solution $X$ on an interval dipath $P$ does not always maintain strong connectivity. This means we cannot expect to reconfigure $\Xend_P$ into a left (or right) stack.
In fact, $\Xend_P$ may not even be reconfigured to be the union of its left and right stacks.
In the following, we show that there exists a sequence of shifts reconfiguring all endpoints onto a support, and eventually achieving strong connectivity.
\begin{lemma}\label{lem_supported-interval-dipaths}
    For any interval dipath $P$, and any solution $X$, there exists a sequence of shifts of $\Xend_P$ resulting in a solution that is supported on $P$.
\end{lemma}
\begin{proof}
    Let $P = [w_1,w_r]$. Consider the endpoints of $\Xend_P$ ordered from left to right along $P$, and recall this ordering is preserved by shifts.
    Throughout the following, let $\Xend^L = \Xend_P^L$ and $\Xend^R = \Xend_P^R$ be the left and right stacks of $X$ on $P$, and define $\Xend^M = \Xend_P \setminus (\Xend^L \cup \Xend^R)$.

    The first step in our reconfiguration is to apply shifts to maximize left and right stacks.
    As long as either the leftmost endpoint of $\Xend^M$ can be shifted left, or the rightmost endpoint can be shifted right, while maintaining a solution, perform any such shift and update $\Xend^L,\Xend^R,\Xend^M$ accordingly.
    Recall that \Cref{lem_no_digon_multicompletion} allows us to ensure the multicompletion resulting from a shift is indeed a solution.
    This process clearly terminates since at any step $|\Xend^M|$ strictly decreases.
    If this ends with $\Xend^M = \emptyset$, we have $\Xend_I = \Xend^L_I \cup \Xend^R_I$. Since $\Xend_I$ corresponds to at most $2|X_P|$ angles, \Cref{lem_stacked-support-bound} yields that it is contained in some $B^L \cup B^R$ for $B^L \in Z^L_P(2|X_P|)$ and $B^R Z^R_P(2|X_P|)$, which concludes this case.

    If the process above doesn't succee, we have $\Xend^M \neq \emptyset$, and we show in the following how to use a sequence of shifts reconfiguring $\Xend^M$ onto a common endpoint.  
    Let $X^M$ be the set of arcs having at least one endpoint in $\Xend^M$, and $D' = D + X - X^M$.
    Note that any terminal component of $D'$ contains an endpoint of $X^M$, and thanks to our maximization of $\Xend^L$ and $\Xend^R$, we obtain the following. 
    \begin{claim}\label{claim:no-terminal-WM}
        No terminal component of $D'$ intersects $W(\Xend^M)$.
    \end{claim}
    \begin{subproof}{the claim}
        Assume otherwise that an arc $a\in X$ has an endpoint at an angle $w_i$ that is contained in a sink component $T'$, without loss of generality.
        Now, \Cref{lem_paths-between-localT} yields a directed path from $w_i$ to $T$ in $D$, meaning $[w_i,w_r] \subseteq T'$.
        Going back to $D+X$, let us re-consider the rightmost endpoint $w_j^e$ of $\Xend^M$, which we could not shift right in the first step.
        Since $j \geq i$, we also have $[w_j,w_r] \in T'$ in $D'$, and a fortiori $[w_j,w_r]$ lies in the same strong component of $D-e$.
        In particular, irrespective of the orientation of $e$, the right shift of $w_j^e$ must result in a strongly connected completion. This contradicts the maximality of $\Xend^R$ and shows the claim.
    \end{subproof}
    

    Rephrasing~\Cref{claim:no-terminal-WM}, all terminal components of $D'$ are incident only to the endpoints of $X^M$ that are not on $P$.
    In particular, we obtain the following corollary.
    \begin{claim}\label{claim:reconf-common-endpoint}
        Substituting all endpoints in $\Xend^M$ with \textsl{any} $w_g \in P$ maintains strong connectivity.
    \end{claim}
    \begin{subproof}{the claim}
        Fix $X^M$, and let $X^M_g$ be the set of arcs obtained from $X^M$ by substituting each endpoint of $\Xend^M$ by $w_g$.
        Because $D' + X^M$ is strongly connected, for any source $S'$ and sink $T'$ of $D'$, $X^M$ contains an arc from $T'$ to (some angle of) $\Xend^M$ and from $\Xend^M$ to $S'$.
        In $D'+X^M_g$, again for any $S',T'$, these arcs have been replaced by arcs from $T'$ to $w_g$ and from $w_g$ to $S'$. In particular in $D'+X^M_g$, any initial sink of $D'$ can reach any initial source through $w_g$, meaning $D'+X^M_g$ is strongly connected.
    \end{subproof}
    
    Now, consider our reconfigured solution from the first step, maximizing $\Xend^L$ and $\Xend^R$, and let us look for an angle $w_g$ to substitute all endpoints of $\Xend^M$ using \Cref{claim:reconf-common-endpoint}. 
    Take the endpoints of $\Xend^M$ from left to right, and successively shift them left until all belong to the (new) left stack of $X$ on $P$.
    This maintains a multicompletion, and now $\Xend^L$ along with the (updated) $\Xend^M$ are by definition in the left stack of $X$ on $P$, meaning they belong to some $B^L_q \in Z^L_P(q)$ by \Cref{lem_stacked-support-bound}.
    Now, this maintains a multicompletion, but may break the strong connectivity of $D+X$. To recover this, we repeatedly apply \Cref{lem_gather} to pairs of consecutive angles in $W(\Xend^{M})$, which eventually shifts all the endpoints of $\Xend^{M}$ onto some common angle $w_g$, which must have been part of $B^L_q$.
    This yields a multicompletion, which is strongly connected by \Cref{claim:reconf-common-endpoint}, and therefore a solution.
    All angles of $\Xend^L$ and $\Xend^M$ now belong to $B^L_q$, and recall all angles of $\Xend^R$ belong to the right stack of $X$ on $P$, therefore to some $B^R_q$, which achieves to show the result.
\end{proof}

Combining the results of this section allows us to restrict our search to supported solutions.
\begin{corollary}\label{cor:existence-supported-solution}
    Any positive instance $(D,k)$ admits a supported solution.
\end{corollary}
\begin{proof}
    Let $D$ be a plane digraph, and $X$ be an arbitrary solution to $D$.
    It suffices to iterate over all faces $F$ of $D$, and (independently) reconfigure $X_F$ into a supported completion.
    Successively consider all local terminals and interval dipaths, and for any such $I$, apply either \Cref{lem_supported-local-terminals} or \Cref{lem_supported-interval-dipaths} to yield a solution that is supported on $I$.
    When handling interval $I$, the only endpoints of $X$ that are reconfigured are $\Xend_I$, meaning $X$ remains supported on $I$ throughout all reconfigurations of $X_F$.
    Then, since the endpoints of $X_F$ are not reconsidered for other faces, we can repeat the process independently for all faces.
\end{proof}

\subsection{Bounding supported completions}\label{ssec_supported_bound}

To find supported solutions, we look for their restrictions to faces of the instance, that is, their supported completions. The following lemma allows us to enumerate those for any face.
\begin{lemma}\label{lem_bound-supported-completions-face}
    For any face $F$ with $\lt$ local terminals, the number of supported completions of $F$ with at most $\ell$ arcs is at most $2^{O(\lt + \ell)}$. Furthermore, those can be enumerated in time $2^{O(\lt + \ell)} n^{O(1)}$.
\end{lemma}
\begin{proof}
    Let $\lt=\lt(F)$, and split the boundary of $F$ into its local terminals and interval dipaths $(I_i)_i$ with $1 \leq i \leq 2\lt$. Note that the angles of these intervals partition $W(F)$.
    There are at most ${{2\lt + 2 \ell} \choose 2\lt} \leq 2^{O(\ell + \lt)}$ ways of choosing the numbers $(p_i)_i$ of (labelled) endpoints of a completion $X$ of size $\ell$ across all the $(I_i)_i$, such that $\sum_i p_i = 2 \ell$. Taking into account for completions with at most $\ell$ arcs factors in another $2^{\ell}$, still yielding $2^{O(\ell + \lt)}$ possible choices of $(p_i)_i$ such that $\sum_i p_i \leq 2 \ell$.
    Then, for each $(p_i)_i$, there are at most $3^2 \cdot 2^{2 p_i}$ ways of taking a support $B_i$ of size $2 p_i$ for $I_i$.
    Indeed, by \Cref{def_supported}, $B_i$ is the union of an element of $Z_I^L(p_i)$ and one of $Z_I^R(p_i)$, which by~\Cref{obs_support-bounded-size-time}, are each of size at most $3 \cdot 2^{p_i-1}$, and can be enumerated in time $2^{O(p_i)} n^{O(1)}$.
    There are thus at most $3^{2 \ell} 2^{\sum_i 2 p_i} = 2^{O(\ell)}$ such choices across all intervals, for any fixed $(p_i)_i$.
    Combining the above, there are $2^{O(\lt + \ell)}$ sets of at most $4 \ell$ angles in $F$ onto which a supported completion can be incident, and those can be enumerated in time $2^{O(\lt + \ell)} n^{O(1)}$.

    Taking any such set $U$, the number of different possible completions $X$ on $U$ is at most the number of distinct outerplanar oriented graphs on $U$.
    Recall that the number of (labelled) outerplanar triangulations on $4 \ell$ vertices is known to be $C_{4 \ell-2} = O(4^{4 \ell}) = 2^{O(\ell)}$, thanks to Walkup~\cite{walkupNumberPlaneTrees1972}.
    These can all be enumerated with $O(\log \log \ell)$ delay, and in particular $2^{O(\ell)} n^{O(1)}$ time, using the algorithm of Bespamyatnikh~\cite{bespamyatnikhEfficientAlgorithmEnumeration2002}.
    Then, to recover all labelled outerplanar oriented graphs, iterate over all triangulations, and for each edge $uv$ within one, choose either non-adjacency, the arc $(u,v)$ or the arc $(v,u)$. Since there are $2 \lt - 3$ edges in each triangulation, there are  $3^{2 \ell} = 2^{O(\ell)}$ such choices for each.
    In total, this yields $2^{O(\ell)}$ different outerplanar oriented graphs on $U$, which can be enumerated in $2^{O(\ell)} n^{O(1)}$ time.
    All in all, taking into account the enumeration of all possible angle sets $U$, there are at most $2^{O(\lt + \ell)}$ supported completions of $F$ with $\ell$ arcs. Moreover, those can be enumerated in time $2^{O(\lt + \ell)} n^{O(1)}$.

\end{proof}

Because there are only a bounded number of local terminals across alternating faces, we can then deduce a bound on the number of supported completions over all of them.
\begin{corollary}\label{cor:branch-alternating}
    Let $(D,k)$ be an instance of \sdpsca with $|\cT(D)| \leq 2k$, and $\mathcal{AF}$ be the set of alternating faces of $D$. Then, the number of supported completions of $\mathcal{AF}$ of size at most $k$ is at most $2^{O(k)}$, and those can be enumerated in $2^{O(k)}n^{O(1)}$ time.
\end{corollary} 
\begin{proof}
    Since $\lt(F) \geq 4$ for every alternating face, $\sum_{F\in \AF(D)} \lt(F) \leq 4 |\cT(D)| - 8$ by~\Cref{lem_bound-lt-alt-oriented} and $\cT(D) \leq 2k$ by assumption, there are at most $2k$ alternating faces $(F_i)$.
    To enumerate all supported completions, we first branch over all possible choices of a number of arcs $k_i$ to complete $F_i$, such that $\sum_i k_i \leq k$, which is $2^{O(k)}$.
    For each choice of such $(k_i)_i$, and each $F_i$, \Cref{lem_bound-supported-completions-face} gives $2^{O(\lt(F_i) + k_i)}$ supported completions of $F_i$ with $k_i$ arcs.
    Combining possible solutions across all of $\AF$ yields $2^{O(\sum_i (\lt(F_i) + k_i))} = 2^{O(k)}$, by \Cref{lem_bound-lt-alt-oriented} and since $\sum_i k_i \leq k$. 
    These can be enumerated in $2^{O(k)}n^{O(1)}$ time thanks to~\Cref{lem_bound-supported-completions-face}.
\end{proof}

For the case of simple faces, $\lt=2$, meaning the bound of~\Cref{lem_bound-supported-completions-face} only depends on the size of the completion. The following will imply that, for our purposes, those are also bounded by an absolute constant.
\begin{lemma}\label{lem_simple_k_max}
     Let $D$ be a plane oriented graph, and $X$ be a minimum solution of $D$, then for any simple face $F$ we have $|X_F|\le 3$.
\end{lemma}
\begin{proof}
    Consider $D' = D+X-X_F$, in which $F$ must be a simple face, and to which $X_F$ is a solution of minimal size.
    We show $D'$ admits a solution of size at most three consisting of arcs embedded in $F$. Substituting $X_F$ for this solution in $X$ yields the result.
    If $F$ is strong in $D'$, then $D'$ is already strongly connected meaning $X_F = \emptyset$, so we can assume $F$ to be simple.
    Let then $S,T$ be the local source and sink of $F$ in $D'$, and $P_1,P_2$ be the interval dipaths from $S$ to $T$.
    We will show that we can always add at most three new arcs in $F$ forming a path from $T$ to $S$, which yields strong connectivity of $V(F)$ and thus a solution of desired size to $D'$.
    
    First, note that if at least one pair of vertices between $S$ and $T$ is a non-arc in $D$, we can add an arc from $T$ to $S$ as a solution of size one.
    Now, we consider $X_F$, which by minimality contains no arc within $T$, nor within $S$, according to \Cref{lem_no_digon_multicompletion}.
    We then know $X_F$ contains an arc $(t,w)$ with $t \in T$ to $w \in P_i$, and one $(w',s)$ with $w' \in P_j$ and $s \in S$ for $i,j \in \{1,2\}$.

    If the addition of $(t,w),(w',s)$ to $D'$ already renders $V(F)$ strong, in particular if $w$ and $w'$ correspond to the same vertex, this yields a solution of size two.
    If $i \neq j$, since there exists an arc between $s$ and $t$ in $D'$, and $D'[V(F)]$ is outerplanar, the arc $(w,w')$ can be added without introducing crossings.
    This can be done along with $(t,w),(w',s)$, which ensures a path from $T$ to $S$, giving a solution of size three.
    Otherwise, assume $i=j=1$ without loss of generality, recall $w \neq w'$, and $w$ must appear after $w'$ following $P_1$ from $S$ to $T$.
    Then, since  $D'[V(F)]$ is outerplanar, there is a non-arc in $D'$ either between $t$ and $w'$, or between $s$ and $w$.
    In the first case, $(t,w')$ along with $(w',s)$ forms a solution of size two, in the second, $(t,w)$ and $(w,s)$ do.
\end{proof}

\Cref{lem_simple_k_max} along with \Cref{lem_bound-supported-completions-face} give an absolute bound for the number of minimum supported completions in simple faces.
\begin{corollary}\label{cor:guess-simple-faces}
    There exists an absolute constant $c_\cS$ such that for any instance $(D,k)$, and any simple face $F$ of $D$, the number of minimal supported completions of $F$ is at most $c_\cS$. Furthermore, those can be enumerated in polynomial time.
\end{corollary}

\subsection{Computing a minimum solution within simple faces}\label{ssec_simple}

At this point, we can branch over all supported completions of size at most $k$ across alternating faces $\AF$.
To obtain an \FPT algorithm, it would now suffice to compute the remaining minimum solution within the (initial) simple faces $\SF$ for each branch, that is, to solve the following problem. 

\defproblem{\tsc{Simple} \spsca}{An instance $(D,k)$ for \spsca, and a (sub)set of simple faces $\cF$ of $D$}{Is there a solution $X$ for $(D,k)$ such that every arc of $X$ is embedded in $\cF$?}

As we have already argued, the existence of multiple (minimal) supported completions inside simple faces means our algorithm should be randomized by picking \say{allowed} completions in each face.
Then, the simplest step would be a straightforward reduction to~\tsc{Weighted SCA}, but the algorithm of~\cite{klinkbyStrongConnectivityAugmentation2021} runs in time $2^{O(k \log k)} n^{O(1)}$.
To obtain a $2^{O(k)} n^{O(1)}$ algorithm, we will instead reduce to the \tsc{Minimum Dijoin} problem.
For a digraph $D$, and any $Y \subseteq A(D)$, we let $\ola{Y} = \{ (v,u) : (u,v) \in Y \}$.
Then, $Y$ is a \emph{dijoin} if $D + \ola{Y}$ is strongly connected.
We will rely on the decision variant of the \tsc{Minimum Dijoin} problem, defined as follows.

\defproblem{\tsc{Dijoin}}{A digraph $D$, and an integer $k$}{Is there some $Y \subseteq A(D)$, with $|Y| \leq k$, such that $D+\ola{Y}$ is strongly connected?}

We now recall the result of Frank~\cite{frankHowMakeDigraph1981}.
\begin{theorem}[Frank '81]\label{thm_dijoin_frank}
    \tsc{Dijoin} is polynomial-time solvable.
\end{theorem}

The following shows how to solve \tsc{Simple PSCA} in $2^{O(k)} n^{O(1)}$ time, by a randomized reduction to \tsc{Dijoin}.
\begin{lemma}\label{lem_simple_psca}
    \tsc{Simple PSCA} admits a $2^{O(k)} n^{O(1)}$ algorithm.
\end{lemma}

\begin{proof}
    Let $(D,k)$ and $\cF$ form an instance of \tsc{Simple PSCA}. 
    We describe a randomized Monte-Carlo algorithm reducing the problem to \tsc{Dijoin}, then show its derandomization.
    Let us first observe that applying~\Cref{lem_supported-local-terminals} and \Cref{lem_supported-interval-dipaths} to a solution of $(D,k)$ within $\cF$ also yields that $(D,k)$ admits a supported solution within $\cF$, and we may therefore only look for these.
    Let $c_{\cS}$ be the maximum number of supported completions given by~\Cref{cor:guess-simple-faces}.
    By the same corollary, for every face $F \in \cF$, we can generate in polynomial time the set $\cX(F)$ of all possible restrictions of a minimum supported solution of $D$ to $F$.
    Then, for each such face $F$, we pick a completion $Y^F \in \cX(F)$ uniformly at random with probability $1/c_{\mathcal{S}}$. 
    
    We now create an instance of \textsc{Dijoin} on an auxiliary digraph $D'$, that will be equivalent to $(D,k)$ admitting a solution using only these \say{allowed} arcs.
    The construction of $D'$ starts from the plane $k+1$-subdivision of $D$, respecting the embedding of $D$.
    Then, consider each $F \in \cF$, any $s \in S$ and $t \in T$, and take the resulting face $F'$ in $D'$.
    Now, for every $e = (u,v) \in Y^F$, we add a vertex $x_{(u,v)}$, along with: the arc $(x_{(u,v)},u)$, a directed path $P_{s,x_e}$ of length $k+1$ from $s$ to $x_{(u,v)}$, and a directed path $P_{x_{(u,v)},v}$ of length $k+1$ from $x_{(u,v)}$ to $v$. This achieves the construction of $D'$, which we stress may not be planar.

    \begin{claim}\label{claim_simple_psca_dijoin}
        $(D',k)$ is positive for \tsc{Dijoin} if and only if $(D,k)$ admits a solution for \tsc{Simple PSCA} contained in $\bigcup_{F \in \cF} Y^F$.
    \end{claim}
    \begin{subproof}{the claim}
        Assume first that we are given some supported solution $X$ to $(D,k)$ with all arcs in $\bigcup_{F \in \cF} Y^F$.
        Consider $D'$, and define $X' = \{ (x_{(u,v)},u) : (u,v) \in X \} \subseteq A(D')$.
        We argue that $D'+\ola{X'}$ is strongly connected by exhibiting a directed path \textsl{in $D'$} from $x$ to $y$ for any $x,y \in V(D)$.
        This will be sufficient as all vertices of $V(D') \setminus V(D)$ admit path from $V(D)$ and one to $V(D)$, by construction.
        Consider then any such $x,y \in V(D)$, let $P$ be the path from $x$ to $y$ in $D+X$.
        We define a path $P'$ from $x$ to $y$ in $D'+\ola{X'}$ by replacing the following arcs of $P$:
        \begin{itemize}
            \item every arc of $P \cap A(D)$ is replaced with its $k+1$-subdivision in $D'$,
            \item every arc $(u,v) \in P \cap X$ is replaced with the path formed by $(u,x_{(u,v)}) \in \ola{X'}$ and $P_{x_{(u,v)},v}$. 
        \end{itemize}
        This achieves to show that $D+\ola{X'}$ is strongly connected and concludes this direction.

        Conversely, assume $D'$ admits a dijoin $X'$ of size at most $k$, and define $X = \{ (u,v) : (x_{(u,v)},u) \in X' \}$.
        Consider any $x,y \in V(D)$, and a directed path $P'$ from $x$ to $y$ in $D' + \ola{X'}$, we build a directed path $P$ from $x$ to $y$ by splitting $P'$ into minimal subpaths from $V(D)$ to $V(D)$, and replacing each with a directed path in $D+X$.
        Consider any such minimal (directed) subpath $P''$ of $P'$ between two vertices of $V(D)$, let $\overline{P''}$ be its underlying undirected path, which we note must already be an underlying path of $D'$.
        
        If $P''$ corresponds to the $k+1$ subdivision of $(a,b)$, we may simply replace it with $(a,b)$.
        Note that $P''$ cannot correspond to the $k+1$-subdivision of $(b,a)$, else every arc of this path would be part of the dijoin, contradicting $|X'| \leq k$. 
        Otherwise, we may assume that there is some face $F \in \cF$, with corresponding $s \in S$ and $t \in T$ chosen by construction, and $(u,v) \in Y^F$, such that $\overline{P''}$ is either between: $s$ and $v$, $s$ and $u$, or $u$ and $v$.
        Let us first deal with the first two cases simultaneously.
        If $P''$ is directed from $s$ to $u$ or $v$, note that \Cref{lem_paths-between-localT} already gives a directed path from $s$ to $u$ or $v$ in $D$, which we use to replace $P''$.
        Moreover, $P''$ could not be directed from $u$ or $v$ to $s$, else all arcs of $P''$ would be part of the dijoin, contradicting its size.
        We are left with the case where $\overline{P''}$ is between $u$ and $v$.
        Then, note again that $P''$ cannot be from $v$ to $u$, else all $k+1$ arcs of $P_{(x_{(u,v)},v)}$ would be part of the dijoin.
        Therefore, $P''$ is from $u$ to $v$, meaning in particular $(x_{(u,v)},u) \in X'$, therefore $(u,v) \in X$, and we may replace $P''$ with $(u,v)$.
        This achieves to show that any subpath $P''$ of $D' + \ola{X'}$ between vertices of $V(D)$ may be replaced with a directed path between the same vertices in $D+X$, and shows $D+X$ is strongly connected.
        The fact that $D+X$ is plane stems from our choice of each $Y^F$ as a plane supported completion, and the fact that it is oriented comes from~\Cref{lem_no_digon_multicompletion}.
    \end{subproof}
    
    We then run the algorithm of Frank~\cite{frankHowMakeDigraph1981} for \tsc{Dijoin} on $(D',k)$, which by the above decides in polynomial time whether $(D,k)$ admits a solution within $\bigcup_{F \in SF} Y^F$.
    We claim that if there is any solution of $(D,k)$ within $\cF$, the algorithm will output \yes with probability at least $1 / {c_{\cS}^{k}}$.
    Indeed, assume this is the case and let $X$ be any minimum supported solution for $(D,k)$ within $\cF$.
    Let $\cF(X) \subseteq \cF$ be the faces in which $X$ embeds an arc, noting $|\cF(X)| \leq k$.
    With probability $1 / {c_{\mathcal{S}}}^{k}$, we have chosen exactly $Y^F = X_F$ for each $F \in \cF(X)$.
    Then, since $X$ is indeed a solution of $(D,k)$ in $\bigcup_{F \in SF} Y^F$, the algorithm for \tsc{Dijoin} will necessarily output \yes.
    This yields a polynomial-time (true-biased) Monte-Carlo algorithm for \tsc{Simple PSCA} which always outputs \no when $(D,k)$ admits no solution, and outputs \yes with probability at least $1 / {c_{\mathcal{S}}}^{k}$ when it does.

    The process above can be derandomized using universal sets.
    We follow the terminology of Ullrich and Vybíral~\cite{ullrichDeterministicConstructionsHighDimensional2022}.
    Here, our universal set will be a family $H$ of functions, each mapping a simple face $F\in \cF$ to an element of $\cX(F)$.
    Formally, we take $H$ to be a $(n,k,c_\cS)$-universal set.
    That is, for any choice of (at most) $k$ faces $F_1,\ldots,F_{k} \in \cF$ among (at most) $n$ faces in $D$, and any choices $Y_i \in \cX(F_i)$ for $i \in [k]$, there is a function $h\in H$ such that $h(F_i)=Y_i$ for every $i$.
    In particular, for one such $h$, we will have $h(F) = X_F$ for every $F \in \cF$, since $|\SF| \leq k$.
    There are explicit constructions of such families $H$ of size $2^{O(k \log {c_\cS})} \log{n} = 2^{O(k)} n^{O(1)}$, computable in linear time in their size, see Theorem 2.5 in~\cite{ullrichDeterministicConstructionsHighDimensional2022} (which relies on~\cite{Naor-splitters}).
    Iterating over the whole universal set yields a factor of $ 2^{O(k)} n^{O(1)}$ for the derandomization, and in turn $(D,k)$ can be decided in time $2^{O(k)} n^{O(1)}$.
    Note that the algorithm of Frank~\cite{frankHowMakeDigraph1981} is constructive, so we may also recover solutions for \tsc{Simple PSCA} using our reduction.
\end{proof}

\subsection{Fixed-parameter tractability}\label{ssec_fpt_psca}

We now combine the branching on alternating faces with the algorithm for \tsc{Simple PSCA} given by~\Cref{lem_simple_psca} to show that \spsca is \FPT.
\fptPsca*

\begin{proof}
    Given an instance $(D,k)$, we begin by partitioning $D$ into strong components in polynomial time, see Tarjan~\cite{tarjanDepthFirstSearchLinear1972}.
    Then, we compute local terminals, as well as the set of alternating faces $\AF$ and simple faces $\SF$, which can also be done in polynomial time.
    If the sum of the number of local terminals across alternating faces is greater than $8(k-1)$, we already output \textsc{No}. Indeed, $(D,k)$ must be a negative instance  by~\Cref{obs_bound_terminals_k} and \Cref{lem_bound-lt-alt-oriented}.
    We may thus assume $\sum_{F \in \AF} lt(F) \leq 8 (k-1)$ from now on.
    
    Leveraging~\Cref{cor:branch-alternating}, we branch over all supported completions $Y^{\AF}$ of size at most $k$ within alternating faces.
    This yields at most $2^{O(k)}$ branches, and iterating over all these completions only adds a $2^{O(k)}n^{O(1)}$ factor to our runtime.
    In each branch, we compute $D' = D + Y^{\AF}$ and $k'=k - |Y^{AF}|$.
    Then, we decide whether $(D',k')$ admits a supported solution within the (initial) simple faces $\SF$ by running our algorithm for \tsc{Simple PSCA} from~\Cref{lem_simple_psca}, with $\cF = \SF$.
    If the branch corresponding to $Y^{\AF}$ outputs \yes for $(D',k')$, letting $Y'$ be a solution, it is clear that $Y^{\AF} + Y'$ is a (supported) solution to $D$.
    Conversely, assume $(D,k)$ admits a minimum solution $X$, and take it to be supported.
    In one of our branches, we will have chosen exactly $Y^{\AF} = X_\AF = \cup_{F\in \AF} X_F$.
    Then, note that $X_{\SF} = \cup_{F\in \SF} X_F$ is a supported solution of $(D',k')$, meaning our algorithm for \tsc{Simple PSCA} $(D',k')$ will have found a (supported) solution of size at most $|X_{\SF}|$, and thus outputs \yes. 
    Combining both the branching and this algorithm yields an algorithm for \spsca running in $2^{O(k)}n^{O(1)}$ time.
     
\end{proof}

\section{NP-completeness}\label{sec_hardness}

In this section, we show~\Cref{thm_np_hardness}, the NP-hardness of the augmentability variant of \spsca, referred to as \spscap, which carries on to \spsca.

\defproblem{\tsc{Plane Strong Connectivity Augmentability} (\textsc{PSCA'})}{A $3$-connected plane oriented graph $D$}{Is there a subset $X \in V(D)^2$ such that $D+X$ is strongly connected, oriented, and plane?}

\noindent
This problem is clearly in \NP as, given a graph $G$ and a set of arcs $X,$ one can verify in polynomial time that $G+X$ is strongly connected and plane.
We show the hardness results by a reduction from \tsc{planar-$3$-SAT}, that is a restriction of \tsc{SAT}, where each clause contains exactly three literals, and the incidence graph of the formula is planar.
This problem has been shown to be \NP-complete by Lichtenstein~\cite{lichtensteinPlanarFormulaeTheir1982}. In Lichtenstein's reduction from \tsc{$3$-SAT}, the size of the instance increases only quadratically. This implies that under ETH, \tsc{planar-$3$-SAT} does not admit $2^{o(\sqrt{n})}$-time algorithm. To obtain the same lower bound under ETH, we have to show that our reduction from \tsc{planar-$3$-SAT} to \textsc{PSCA'} is linear, in the sense that a \tsc{planar-$3$-SAT} with $n$ variable leads to an instance with $O(n)$ vertices. This is proved in~\Cref{thm_psca_prime_hardness}, but let us start with some terminology and a description of the reduction.

Consider an instance $\Phi = \bigwedge_{j=1}^m \phi_j$ of \tsc{planar-$3$-SAT}, on variables $x_1,...,x_n,$ which we consider with any fixed plane embedding. We assume without loss of generality that the incidence graph of $\Phi$ is connected.
We construct a plane oriented graph $D$ such that $\Phi$ is satisfiable if and only if $D$ is positive for \spscap.
Our gadgets are plane and oriented, and eventually, none will be embedded in an internal face of another.
Therefore, while the outerface of any gadget may \say{interact} with the rest of the instance, logical choices in the reduction correspond to completions within inner faces.
We formalize these \say{internal} completions in the following, introducing concepts allowing us to constrain possible completions given by positive reduced instances of \spscap.

\paragraph{Internal completions.}
In the following, we consider a plane subdigraph $H$ of some instance $D$ to \spscap, which will later correspond to a gadget, and introduce some terminology about completions within $H.$
An \emph{internal} completion of $H$ is a completion such that all arcs are embedded in the inner faces of $H.$
A terminal $x$ of $H$ is \emph{internal} if it is not incident to the outerface. Then, any valid solution for $D$ must add an arc incident to $x,$ towards $x$ if $x$ is a source and from $x$ if it is a sink.
Note that such an arc will necessarily be embedded in an internal face of $H.$
We say a completion $X$ of $H$ is \emph{valid} if it is internal, and all terminal components of $H+X$ lie on the outerface.

Given two completions $Y,X$ of $H,$ we say $Y$ (strictly) \emph{dominates} $X$ if the partition of $H+X$ into strong components is a (strict) refinement of the one of $H+Y.$ Two completions dominating each-other are called \emph{equivalent}, and admit the same partition.
Then, a completion is \emph{maximal} when it is not dominated by another, but note that there may exist multiple maximal completions.
We will use the following observation throughout our reduction, allowing us to assume the completion within each gadget is a maximal one. 
\begin{observation}\label{obs_max-completion}
    Let $D$ be a positive instance for \spscap admitting a solution $Y,$ and $H$ be a (plane) subidgraph of $D.$
    Then, the restriction $Y_H$ of $Y$ to $H$ is valid, and replacing $Y_H$ in $Y$ with a dominating valid completion still yields a solution.
\end{observation}

\paragraph{Sketch of the reduction.}
We begin in \Cref{ssec_gadgets} by defining the literal, variable, and clause gadgets used in our reduction.
Literal gadgets admit two maximal valid completions, corresponding to the positive and negative assignments.
The gadget for a variable $x$ consists in a combination of as many literal gadgets as the number of occurrences $n_x$ of $x$ in $\Phi.$ We show that any maximal valid completion of a variable gadget is either the positive or negative completion of all of its literals.
As a result, in a positive instance, the completion of the variable gadget consists in a single sink component containing all vertices but $n_x$ \say{top} sources, lying on the outerface.
Clause gadgets are inner triangulations, admitting no internal completion.
Our construction, described in \Cref{ssec_construction}, shows how to use these clause gadgets to join sources of the variables (literals) involved into a single source component.
We show the soundness of our reduction in \Cref{ssec_soundness}, achieving to prove \Cref{thm_np_hardness}.

\subsection{Gadgets}\label{ssec_gadgets}

Consider a variable $x$ appearing $n_x$ times in $\Phi.$
We will build the variable gadget for $x$ by combining $n_x$ literal gadgets, each identified and connected to the gadget of their corresponding clause.

\subsubsection{Literals}

To construct a literal gadget $L$, as depicted in \Cref{fig:literal-def}, we start with an alternating cycle of length $8,$ embedded planarily.
We choose an arbitrary sink to be the bottom sink and label it $b,$ then label remaining left, top and right sinks in clockwise order as $l,t,r.$
We label the sources in-between two sinks as $bl$ (bottom left), $tl$ (top left), $tr$ (top right) and $br$ (bottom right), according to the sinks they are adjacent to.
Create twins $b',l',t',r'$ for each sink and  embed them on the outer face of the octagon, then, add an arc from the twin to the original sink.
Then, add the arc $(t,b)$ within the octagon, leaving only $b,l,r$ as sinks. This results in two $5$-faces, which are the only ones that may be completed.
\begin{figure}
    \centering
    \includegraphics{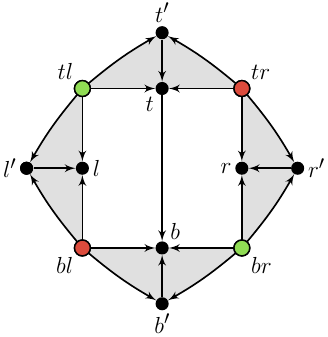}
    \caption{A literal gadget $L$, with internal sinks $b,l,r,$ bottom sources $bl,br$ and  top sources $tl,tr.$}
    \label{fig:literal-def}
\end{figure}

We say that a completion $X$ of a gadget \textit{hits} a source $s$ if it contains an arc towards $s,$ then $X$ hits a sink $t$ if it contains an arc from $t.$ 
Sinks of a literal gadget are internal to it and  they will remain internal in the variable gadget, forcing them all to be hit by a valid completion.
In the variable gadget, bottom sources of each literal will also be internal, which will force at least one to be hit.
The following lemma shows that these two constraints are enough to restrict our completions to exactly two maximal ones.
\begin{lemma}\label{lem_literal-completion}
    Any valid completion of the literal gadget $L$ hitting $br$ is either: 
    \begin{itemize}
        \item the \emph{positive completion} $Y=\{(r,t),(t,br),(l,b),(b,tl)\},$ yielding sources $bl,tr$ and only one other (sink) strong component in $L+Y$, or
        \item the completion $Y=\{(r,t),(t,br),(b,l),(l,t)\},$ that is dominated by the positive completion.
    \end{itemize}
    
    Any valid completion hitting $bl$ is either:
    \begin{itemize}
        \item the \emph{negative completion} $Y=\{(l,t),(t,bl),(r,b),(b,tr)\},$ yielding sources $br,tl$ and only one other (sink) strong component in $L+Y$, or
        \item the completion $Y= \{(l,t),(t,bl),(b,r),(r,t)\},$ that is dominated by the negative completion.
    \end{itemize}

    In all cases, all non-sources vertices form a sink component.
\end{lemma}
\begin{proof}
    Let $Y$ be any valid completion hitting at least one bottom source $bl,br.$
    Since $l,b,r$ are internal sinks, they must be hit by $Y$ and  observe that no arc hitting those can be added towards $bl$ or $br.$ This means that the arcs hitting $l,b,r$, and one of $bl,br$ are distinct. Since at most two arcs can be added in each 5-face, $|Y|=4$ and exactly two arcs of $Y$ are embedded in each of the left and right faces.

    Assume first that $Y$ hits $br.$
    The arc hitting $br$ is exactly $(t,br) \in Y,$ as otherwise the internal sink $r$ could not be hit. Thus, to hit $r,$ the (unique) other arc of $Y$ in the right $5$-face is $(r,t).$
    If the arc of $Y$ hitting $b$ is $(b,tl),$ this yields the positive completion $\{(r,t),(t,br),(l,b),(b,tl)\}$ (see the second completion in \Cref{fig:literal}). Otherwise, it must be $(b,l),$ which yields completion $\{(r,t),(t,br),(b,l),(l,t)\}.$
    In both cases, since $bl,tr$ are not hit, they are sources of $L+Y$ and  in the second case, $tl$ must also be. 

    The case is symmetrical if $Y$ hits $bl,$ which forces arcs of $Y$ added in the left case to be $(t,bl)$ and $(l,t).$
    Then, according to the arc hitting $b$ in the left case, we either obtain the negative completion $Y=\{(l,t),(t,bl),(r,b),(b,tr)\},$ yielding sources $br,tl$ (third completion in \Cref{fig:literal}). Or, we obtain $Y=\{(l,t),(t,bl),(b,r),(r,t)\},$ yielding sources $br,tl,tr$ (first completion in  \Cref{fig:literal}).
\end{proof}

\begin{figure}[t]
    \centering
    \includegraphics[scale=1.3,page=1]{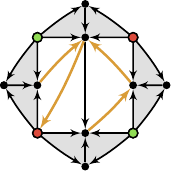}
    \hspace{1.2cm}
    \includegraphics[scale=1.3,page=2]{literal.pdf}
    \hspace{1.2cm}
    \includegraphics[scale=1.3,page=3]{literal.pdf}
    \caption{A literal gadget $L$ along with three completions.
	The first completion is valid and hits a bottom source, but is dominated by the third.  
	The second and third completions correspond to the positive and negative (valid) completions respectively. These hit exactly one bottom source and  form a sink component containing all vertices but the two antipodal sources which are not hit.}
	\label{fig:literal}
\end{figure}

While a positive reduced instance only implies that the completion within literals is valid, the following variable gadget ensures that their completion must also be maximal, which forces either the positive or negative completion. 

\subsubsection{Variables}

Let us formally describe the construction of the gadget $G_x$ for a variable $x$ appearing $n_x$ times in $\phi.$ Note that we may assume $n_x \geq 2$ as otherwise the \tsc{$3$-SAT} instance can be reduced by removing $x$ and the associated clause.
Take $n_x$ copies of a literal gadget, $L_1,...,L_{n_x},$ with corresponding vertices now labelled with subscripts $i \in [1,n_x].$
Considering indices $i$ modulo $n_x,$ we build the variable gadget for $x$ as follows:
\begin{itemize}
    \item For $i \in [1,n_x],$ identify sources $br_i \in V(L_i)$ and $bl_{i+1} \in V(L_{i+1}),$
    \item For $i \in [1,n_x],$ identify $r'_i$ with $l'_{i+1},$
    \item Finally, identify all $(b'_i)_i$ into a common vertex $b'.$ 
\end{itemize}
We embed the resulting gadget canonically, such that the (antidirected) cycle formed by $(b_i)_i$ and $(br_i)_i$ is plane and  the embedding of any $L_i$ is preserved.
The final construction is depicted in \Cref{fig:variable-gadget} with a completion which will correspond to a positive assignment.

\begin{figure}[t]
    \centering
    \includegraphics[scale=1]{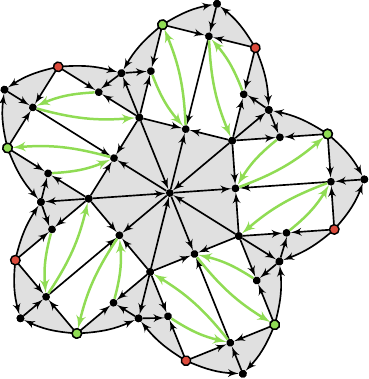}
    \caption{The gadget $G_x$ for variable $x,$ appearing in $5$ clauses, with a maximal completion corresponding to the positive assignment shown in green. This yields a valid completion with top-right vertices of each literal being sources, while all other vertices form a sink component.}\label{fig:variable-gadget}
\end{figure}

\begin{lemma}\label{lem_completionVariable}
    A variable gadget $G_x$ admits exactly two maximal valid completions:
    \begin{itemize}
        \item The \emph{positive completion}, with source vertices $(tr_i)_i$ and sink component $Y_x = V(G_x) \setminus \{tr_i : i\}.$
        \item The \emph{negative completion}, with source vertices $(tl_i)_i$ and sink component $Y_x = V(G_x) \setminus \{tl_i : i\}.$ 
    \end{itemize}
\end{lemma}

\begin{proof}
    Consider a variable gadget $G_x,$ comprised of literal gadgets $(L_i)_i$ for $i \in [1,n_x].$
    Let $Y$ be a maximal valid completion of $G_x$ and  $Y_i = Y \cap V(L_i)^2$ be its restriction to the literal gadget $L_i.$

    The internal sinks of $G_x$ are exactly $(l_i,b_i,r_i)_i,$ for $i \in [n_x],$ meaning each $Y_i$ is a valid completion of $L_i.$
    The internal sources of $G_x$ are exactly $(br_i)_i,$ which must all be hit by arcs of $Y$ within the literal gadgets, by construction.
    Then, we know from \Cref{lem_literal-completion} that at most one such arc may belong to a single $Y_i.$ So each $Y_i$ contains exactly one arc incident to $br_i$ or $bl_i=br_{i-1}.$
    We claim $Y_i$ must also be maximal for $L_i.$ Assume it is not, then it must be one of the two completions dominated by the positive and negative ones, thanks to \Cref{lem_literal-completion}. In both cases, $tl_i$ and $tr_i$ are sources in $L_i+Y_i,$ thus also in $G_x+Y.$ Then, substituting $Y_i$ with the completion dominating it would yield a completion of $G_x$ dominating $Y,$ which contradicts its maximality.
    Therefore, each $Y_i$ is either the positive or the negative completion of $L_i.$ 

    Now, assume for some $i$ that $Y_i$ is the positive completion of $L_i,$ that is, the (unique) valid solution hitting $br_i.$
    Since $bl_i$ and $br_{i-1}$ are identified, $bl_i$ must now be hit by $Y_{i-1},$ meaning $Y_{i-1}$ is the positive completion of $L_{i-1}$ as well.
    Applying the above inductively yields that each $Y_j$ is exactly the positive completion of $L_j$ for all $j \in [n_x].$
    Conversely, if for some $i,$ $Y_i$ is the negative completion of $L_i,$ then this is the case for all $i.$
    
    In turn, $Y$ must be one these two completions, which we call \emph{positive} and \emph{negative} respectively.
    In both cases, $Y$ contains a directed cycle containing all internal sources. If $Y$ is positive, it can be obtained \say{clockwise} by concatenating directed paths $(bl_i,b_i,tl_i,t_i,bl_{i+1}).$ If $Y$ is negative, it can be found \say{anti-clockwise} by concatenating directed paths $(br_i,b_i,tr_i,t_i,bl_{i-1}).$
    Applying \Cref{lem_literal-completion}, $L_i+Y_i$ consists in a sink component containing all but source vertices $\{bl_i,tr_i\}$ in the positive case or  $\{br_i,tl_i\}$ in the negative case.
    Combining the two observation, $G_x+Y$ contains sources $\{tr_i : i\}$ and a (unique) sink component $Y_x = V(G_x) \setminus \{tr_i : i\}$ in the positive case. And in the negative case, it contains sources $\{tl_i : i\}$ and a (unique) sink component $Y_x = V(G_x) \setminus \{tl_i : i\}.$
    This shows that both completions are indeed valid and  the unique such ones that are maximal.
\end{proof}

\subsubsection{Clauses}


We now describe the clause gadgets, which we will use in \Cref{ssec_construction} to join top sources of corresponding variable gadgets.
Consider a clause $\phi = l(x) \vee l(y) \vee l(z)$ of $\Phi,$ we build the gadget $C_{\phi}$ with vertices assigned to $x,y,z,$ irrespective of the corresponding literal.
The clause gadget consists of $15$ vertices, it is an inner-triangulated plane oriented graph, depicted in \Cref{fig:clause-gadget} and defined as below.
We start with vertices $v_x,v_y,v_z,$ forming a facial directed triangle $\ora{C}_{\phi}.$
Then, for each pair of $v_x,v_y,v_z,$ add a common out-neighbour such that the resulting triangle is also facial, call those $m_{xy},m_{yz},m_{zx}.$
Now, we add add out-neighbours $u_x,w_x$ to $v_x$; $u_y,w_y$ to $v_y,$ and $u_z,w_z$ to $v_z.$
Then, create directed $4$-cycles $\ora{C}_{xy} = (u_x,m_{xy},w_y,m_{yx}),$ $\ora{C}_{yz} = (u_y,m_{yz},w_z,m_{zy}),$ and $\ora{C}_{zx} = (u_z,m_{zx},w_x,m_{xz}),$ and embed them such that the outer boundary is $(u_x,v_x,w_x,m_{xy},u_y,v_y,w_y,m_{yz},u_z,v_z,w_z,m_{zx}).$ Finally, add arcs $(m_{xz},m_{zx}),(m_{yx},m_{xy}$ and $(m_{zy},m_{yz}).$
Note that the resulting gadget consists of the source component $\ora{C}_{\phi},$ consisting of $\{v_x,v_y,v_z\},$ and its sink components are $\ora{C}_{xy},\ora{C}_{yz},\ora{C}_{zx}.$

\begin{figure}[H]
    \centering
    \includegraphics[scale=1]{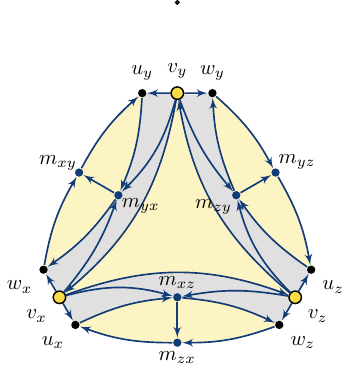}
	\caption{A clause gadget $C_{\phi}$ for some $\phi = l(x) \vee l(y) \vee l(z)$. The central source component is highlighted in yellow, and consists of vertices $\{v_x,v_y,v_z\}$. These are later identified to sources of the corresponding literal gadgets, meaning one literal will have to be satisfied. The remaining three sink components, also shown in yellow, will be identified with literal gadgets as to not be terminals anymore.}
	\label{fig:clause-gadget}
\end{figure}

\subsection{Construction}\label{ssec_construction}

We are now ready to describe our construction of an instance $D$ to \spscap, equivalent to the \tsc{planar-$3$-SAT} formula $\Phi = \bigwedge_{j \in [m]} \phi_j$ on variables $(x_i)_{i \in n}.$
We begin by embedding the gadgets $G_{x_i}$ for variables $(x_i)_i,$ respecting the embedding of $\Phi.$
For each variable $x,$ recall $G_x$ is formed by $n_x$ literal gadgets, which we (re)label $L_{x,\phi}$ for any $\phi = \phi_j$ containing $x.$
That is, the (clockwise) order of literal gadgets inside $G_x$ follows the order of incidences of (literals of) $x$ in the corresponding clauses of $\Phi.$
Likewise, when writing $\phi =  a \vee b \vee c,$ we assume $a,b,c$ are ordered clockwise according to their incidence on the corresponding variables in $\Phi.$

Now, we show how to join variable gadgets with the corresponding clause gadgets, as depicted in \Cref{fig:reduction}.
Consider any clause $\phi = l(x) \vee l(y) \vee l(z),$ with $x,y,z$ appearing clockwise in $\Phi,$ and let $C_{\phi}$ be the clause gadget of $\phi.$ Let $L_{x,\phi},$ $L_{y,\phi}$ and $L_{z,\phi}$ be the corresponding literal gadgets of $G_x,G_y,G_z.$
For any variable $q \in \{x,y,z\},$ we identify a top source and two outer vertices from $L_{q,\phi}$ with vertices $u_q,v_q,w_q$ of $C_{\phi}$ according to $l(q)$ as follows.
\begin{itemize}
    \item If $l(q) = q,$ we identify top left source $tl_{w,\phi}$ with $v_q,$ then $t'_{q,\phi}$ with $u_q,$ and $l'_{q,\phi}$ with $w_q.$  
    \item If $l(q) = \overline{q},$ we identify top right source $tr_{w,\phi}$ with $v_q,$ then $r'_{q,\phi}$ with $u_q,$ and $t'_{q,\phi}$ with $w_q.$ 
\end{itemize} 
Then, we preserve the (relative) embeddings of $m_{xy},m_{yz},m_{zx}$ and $m_{yx},m_{zy},m_{xz}$ from $C_{\phi}$ given by the construction of the clause gadget.
Finally, we add arcs entering the remaining top sources in each literal gadget.
We do so in such a way that no matter the maximal valid completions of the variable gadgets, the sink components of each are eventually part of the same strongly connected component. Here, we assume that the incidence graph of $\Phi$ is connected, as otherwise we would deal with each connected component independently.
To connect these top sources, we consider variable gadgets clockwise around             $C_{\phi},$ and for any consecutive          variables $p,q$ in $\{x,y,z\},$ we add the following.
\begin{itemize}
    \item If $l(p)=\overline{p}$ and $l(q) = \overline{q},$ we add the arc from $r'_{q,\phi} \in V(L_{q,\phi})$ to top source $tl_{p,\phi} \in V(L_{p,\phi}).$ See the adjacency between gadgets of $y=p$ and $z=q$ in \Cref{fig:reduction}.
    \item If $l(p)=p$ and $l(q) = q,$ we add the arc from $l'_{p,\phi} \in V(L_{p,\phi})$ to top source $tr_{q,\phi} \in V(L_{q,\phi}),$ symmetrically to the above.
    \item If $l(p)=\overline{p}$ and $l(q) = q,$ we add the arc from $t'_{p,\phi} \in V(L_{p,\phi})$ to top source $tr_{q,\phi} \in V(L_{q,\phi}),$ as well as the arc from $tr_{q,\phi}$ to $tl_{p,\phi}.$ See the adjacency between gadgets of $z=p$ and $x=q$ in \Cref{fig:reduction}.
    \item If $l(p)=p$ and $l(q)=\overline{q},$ we add no arcs. See the adjacency between gadgets of $x=p$ and $y=q$ in \Cref{fig:reduction}.
\end{itemize}
We apply the construction above to any variable/clause incidence.
Since $\Phi$ is planar, these operations can be done while preserving a plane digraph.
Moreover, the instance remains oriented, as the arcs added (directly) between variable gadgets as above will not be reconsidered for a different clause.

\begin{figure}
    \centering
	\includegraphics[scale=0.9]{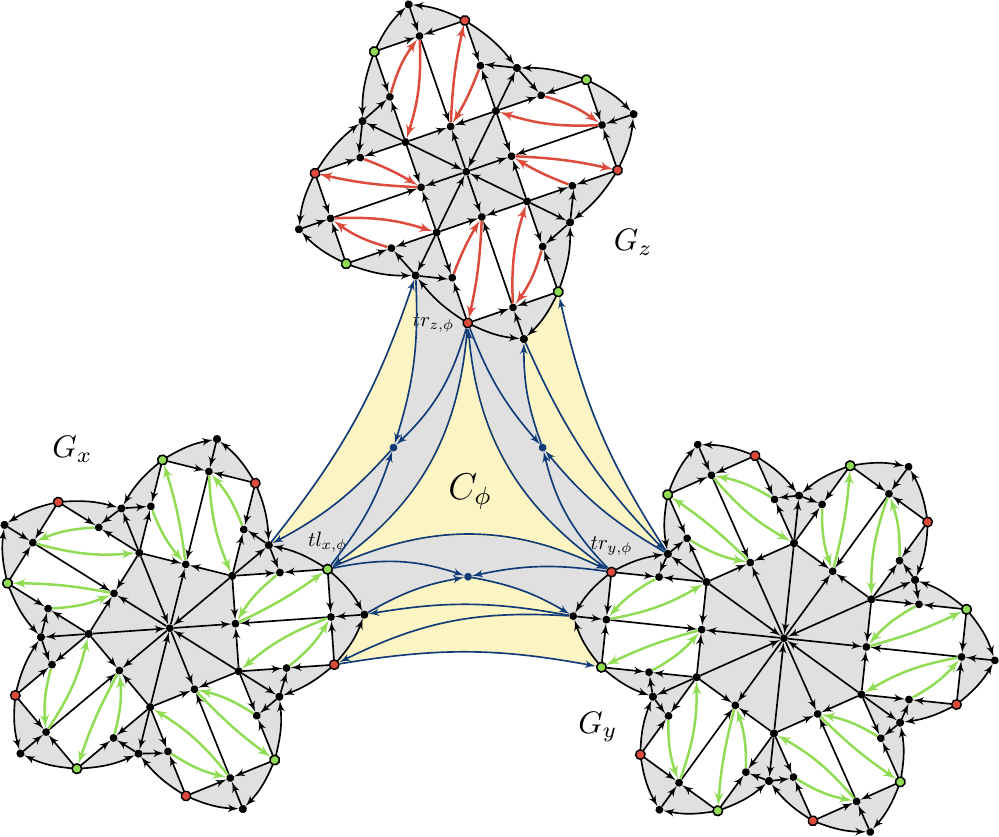}
	\caption{A clause gadget corresponding to $\phi = x \vee \overline{y} \vee \overline{z},$ with variable gadgets for $x,$ $y$ and $z$ (occurring $5,$ $4$ and  $6$ times in $\Phi$).
	Top source of literal gadgets are now identified to the (central) source component of the clause. Here, $x$ and $y$ are completed positively and negatively, satisfying $\phi,$ meaning the (initial) source of the clause is now part of the same strongly connected component as sink components of $x$ and $y.$}
	\label{fig:reduction}
\end{figure}

\subsection{Soundness}\label{ssec_soundness}

We are now ready to show the hardness of \spscap through the following, which we recall also implies the $2^{O(\sqrt{k})}n^{O(1)}$ lower bound under ETH, and thus~\Cref{thm_np_hardness}.

\begin{theorem}\label{thm_psca_prime_hardness}
    \spscap admits a linear reduction from \tsc{planar-$3$-SAT}.
\end{theorem}
\begin{proof}
    Consider a \tsc{planar-$3$-SAT} instance $\Phi$ on variables $(x_i)_i$ and clauses $(\phi_j)_j,$ which we may take to be connected.
    Let $D$ be the plane oriented graph built from $\Phi$ in~\Cref{ssec_construction} be the instance of \spscap. It is clear that $D$ has linear size in $\Phi$ by construction.

    \paragraph{The initial instance is positive implies the reduced instance is.}

    Assume first that $\Phi$ is a positive instance for \tsc{$3$-SAT} and  let us show that $D$ admits a strongly connected augmentation preserving planarity (again, without digons).
    This augmentation is given by completing variable gadgets of $(x_i)_i$ with positive and negative completions according to any assignment satisfying $\Phi.$ Let $X$ be this completion, as defined in \Cref{lem_completionVariable}, and consider the resulting (embedded) graph $D+X.$
    By construction, $D+X$ is plane and oriented, and we now show it is strongly connected by progressively building a set $Y \subseteq V,$ ending with $Y=V,$ and arguing for the strong connectivity of $Y$ at each step.

    In $D+X,$ for any variable $x,$ the corresponding gadget $G_{x}$ has a single sink component $Y_{x}$, according to \Cref{lem_completionVariable}. This component contains all vertices except: $(tr_{x,\phi})_{\phi}$ if $x$ has been completed positively, or $(tl_{x,\phi})_{\phi}$ if it has been completed negatively; for $\phi$ ranging through clauses involving $x.$ 
    
    For any three variables $x,y,z$ appearing in a clause $\phi,$ $Y_x \cup Y_y \cup Y_z$ are part of the same strongly connected component in $D+X.$ This is because of the $4$-cycles $C_{\phi}$ added between vertices ($t',$ $l'$ or $r'$) of $Y_x$ and $Y_y,$ $Y_y$ and $Y_z,$ as well as $Y_z$ and $Y_x.$
    Because the incidence graph of $\Phi$ is connected, letting $Y = \bigcup_i Y_{x_i},$ to which we add vertices of all directed $4$-cycles within clauses, yields that $Y$ is strongly connected.

    Observe that any sink (respectively source) component of $D+X$ must contain a sink (respectively source) component of $D.$ First, all sinks of variable gadgets are part of $Y,$ and recall this is also true for the sink components (directed $4$-cycles) of clause gadgets $C_\phi.$ 
    Therefore, if $D+X$ is not strongly connected, it admits an unique sink component which contains $Y,$ and any vertex outside $Y$ admits a directed path towards $Y.$
    
    Then, the observation above applied to source (components) yields that any source in $D+X$ must contain a top source belonging to a literal gadget, because all vertices of clause sources $\ora{C}_{\phi}$ have been identified with those.
    Now, for any clause $\phi$ on $x,y,z,$ our construction adds arcs from $Y_x \cup Y_y \cup Y_z$ to the top sources of $L_{x,\phi},L_{y,\phi},L_{z,\phi}$ that have not been identified to the source component $\ora{C}_{\phi}$ of $\phi.$ In turn, these (initial) top sources are part of the same component as $Y,$ and we add them to $Y.$
    Therefore, the only vertices of $L_{x,\phi},L_{y,\phi},L_{z,\phi}$  which may be contained in a source of $D+X$ are the three (initial) top sources corresponding to the literal appearing in $\phi,$ which are now joined into a directed triangle $\ora{C}_{\phi}.$
    Assume without loss of generality that in our assignment for $\Phi,$ clause $\phi$ is satisfied by variable $x.$ If $l(x)=x$ in $\phi,$ $\ora{C}_{\phi}$ contains $tl_{x,\phi},$ which is part of $Y$ since $x$ has been completed positively. Conversely if $l(x) = \overline{x},$ $\ora{C}_{\phi}$ contains $tr_{x,\phi},$ which is part of $Y.$ In any case, $C_{\phi}$ is part of the same strongly connected component as $Y.$ Applying this reasoning to all clauses of $\Phi$ yields that $D+X$ is strongly connected, achieving to show this implication.

    \paragraph{The reduced instance is positive implies the initial instance is.}

    Conversely, assume $D$ is a positive instance of \spscap, and take $X$ to be a solution.
    According to \Cref{obs_max-completion}, we may assume the restriction of $X$ to each variable gadget $G_x$ is a maximal valid completion of the gadget.
    Therefore, we may apply~\Cref{lem_completionVariable} to yield that the restriction of $X$ to each variable gadget $G_x$ is either its positive or negative completion. We assign boolean values to $(x_i)_i$ according to the completions of their gadgets, and show this satisfies $\Phi.$

    Towards this, consider any clause $\phi = l(x) \vee l(y) \vee l(z),$ and let us show it is satisfied.
    Let $C_{\phi}$ be its corresponding gadget and consider its (initial) source component $\ora{C}_{\phi}.$ Since $D+X$ is strongly connected, $X$ must contain an arc $e$ towards a vertex of $\ora{C}_{\phi} = (v_x,v_y,v_z).$
    Assume without loss of generality that $e$ is towards $v_x.$
    If $l(x) = x$ in $\phi,$ our construction identified $v_x$ with top source $tl_{x}$ of $L_{x,\phi}.$
    Recalling~\Cref{lem_completionVariable}, this means that the restriction of $X$ to $G_x$ is its positive completion. In turn, $x$ has been set to true in our assignment for $\Phi,$ and $\phi$ is satisfied.
    If instead $l(x) = \overline{x},$ the restriction of $X$ to $G_x$ is its negative completion, $x$ has been set to false, and $\phi$ is also satisfied.
    This achieves to show that our assignment satisfies $\Phi,$ and concludes the reduction.   
\end{proof}

\bibliographystyle{plainurl}
\bibliography{biblio}

\newpage
\appendix

\section{Proofs of \Cref{sec_dirPSCA}}

\subsection{Reducing to DAGs}\label{apdx:reduc-DAG}

Throughout this subsection, we relax all notions of~\Cref{sec_dirPSCA} to multidigraphs in place of digraphs, in particular the instances of \dirPSCA.
We moreover consider a \emph{condensation} of $D,$ $D^c,$ as a plane acyclic multidigraph obtained by contracting iteratively every non-loop arc belonging to a strong component of $D$ (preserving multi-arcs and loops). 
Note that the embedding of $D^c$ can vary according to the contraction sequence, but not the underlying multidigraph. 

\begin{restatable}{lemma}{reducAlmostDAGs}\label{lem_reduc-almost-DAGs}
    An instance $(D,k)$ of \dirPSCA admits a solution if and only if $(D^c,k)$ does, for any condensation $D^c.$
    Moreover, any solution for $D^c$ can be transformed into a solution for $(D,k)$ in polynomial time. 
\end{restatable}
\begin{proofof}{\Cref{lem_reduc-almost-DAGs}}
Let us consider a plane multidigraph $D$, and let $D'$ be obtained from $D$ by contracting a non-loop arc $(u,v)$ contained in a strong component $C$ of $D.$ Let $w$ be the vertex of $D'$ resulting from the merging of $u$ and $v.$
It suffices to show how to transform a solution $X$ for $D$ into a solution $X'$ for $D',$ with $|X|=|X'|$ and  vice-versa. 

Given a solution $X$ for $D,$ we construct $X'$ from $X$ by keeping every arc not incident to $u$ or $v$, and replacing the $u$ or $v$ endpoint of the other arcs with $w.$ It is clear that $D'+X'$ is strong, and we can moreover embed it to be plane.

Given a solution $X'$ for $D',$ let us construct $X$ as follows.
Let us first distinguish some faces of $D.$ Let $\widehat{u}_1,\widehat{u}_2,\widehat{v}_1,\widehat{v}_2$ be the angles\footnote{Angles are defined in \Cref{ssec_definitions}} at $u$ or $v$ incident to $(u,v).$
Let $F_u(D),F_v(D)$ be the faces of $D$ that are incident to $u$ or $v,$ respectively, at an angle distinct from $\widehat{u}_1,\widehat{u}_2,\widehat{v}_1,\widehat{v}_2.$ Note that as the faces of $D$ and $D'$ have a natural bijection, $F_u(D),F_v(D)$ also denote sets of faces in $D'.$
To build $X,$ keep every arc that is not incident to $w$ and  for the other arcs $e,$ replace their endpoint $w$ by $u$ (resp. $v$), if $e$ is drawn in $F_u(D)$ (resp. if $e$ is not drawn in $F_u(D)$). 
The planarity of $D+X$ is clear and  for the strong connectivity, note that $u$ and $v$ are connected in both direction, in $D$ and  thus also in $D+X.$ Indeed, there is a arc from $u$ to $v$ and  as they belong both to $C$ there is a dipath from $v$ to $u.$ This implies that any dipath going though $w$ in $D'+X'$ can be rerouted through $u$ and $v.$
\end{proofof}

Let us now recall and prove \Cref{lem_reduc-DAG}.

\reducDAG*

\begin{proofof}{\Cref{lem_reduc-DAG}}\label{pf:reduc-DAG}
    For a polynomial-time Turing reduction let us see how to solve an arbitrary instance $(D,k)$ through a solution for polynomially many acyclic instances $(D_i,k_i)$ of similar sizes, in polynomial time.
    
    From \Cref{lem_reduc-almost-DAGs}, we know that any instance of \dirPSCA polynomially reduces to an instance $(D,k)$ that is a plane DAG with (possible) loops.
    We hence just have to deal with getting rid of the loops.
    Before going further, note that there is no point of having a loop in a solution $X,$ so faces incident to only one vertex should not contain any arc of a minimum solution $X.$
    
    For any loop $\ell=(v,v)$ of $D$, let $V_{in},V_{out}$ be the sets of vertices drawn inside or outside $\ell,$ respectively, those sets being possibly empty. Clearly, any solution $X$ for $D$ partitions into two solutions $X_{in},X_{out}$ (according to the position of the arcs with respect to $\ell$) for the subgraphs $D_{in} = D[V_{in}\cup\{v\}]$ and $D_{out} = D[V_{out}\cup\{v\}].$
    Indeed, any dipath with both ends in $D_{in}$ (resp $D_{out}$) cannot go through a vertex of $V_{out}$ (resp. $V_{in}$).
    Conversely, taking the union of two solutions $X_{in},X_{out}$ for $D_{in}$ and $D_{out},$ leads to a solution for $D.$ Indeed, this follows from the fact that every vertex of $V_{in}$ (resp. $V_{out}$) has a dipath from $v$ and a dipath towards $v$ in $D_{in}$ (resp. $D_{out}$). 
    
    Finally, observe that an arc set makes $D'_{in} = D_{in}-\{\ell\}$ (resp. $D'_{out} = D_{out}-\{\ell\}$) plane and strong if and only if it makes $D_{in}$ (resp. $D_{out}$) plane and strong.
    Hence, $(D,k)$ is a positive instance if and only if there exists $k_{in},k_{out}$ such that $k=k_{in}+k_{out}$ and  such that $(D'_{in},k_{in})$ and $(D'_{out},k_{out})$ are both positive instances. Actually our algorithm does not consider these $k+1$ pairs of instances, but rather recursively splits $D'_{in}$ and $D'_{out}$ along their loops until reaching loopless subgraphs.
    For each such splitting, we have that $|A(D)| = |A(D'_{in})| +|A(D'_{in})| +1$, so the recursive process above leads to at most $|A(D)|$ acyclic inputs.
    For each of these acyclic instances we try all the $k+1$ values in $[0,k].$ Then, we just have to see if the minimal number of arcs needed for the combination of these acyclic instances sums to at most $k$.
\end{proofof}

\subsection{Bounding local terminals in alternating faces}
\label{ssec_dir_bound_local}

We now wish to control the number of local terminals in terms of $|\cT(D)|,$ both within each face and  across the instance.
One cannot hope for a bound on the global number of local terminals.
Indeed, taking the union of any arbitrarily many directed paths, identifying all sources and  all sinks, gives an instance with two terminals and solution size $1,$ but with and an arbitrarily large number of simple faces.
We show here that such a bound is still attainable for alternating faces.
First, we need the following lemma, which can also be derived from a result of Guattery and Miller~\cite{guatteryContractionProcedurePlanar1992}.
\begin{restatable}[Theorem 2.1 in ~\cite{guatteryContractionProcedurePlanar1992}]{theorem}{boundSnDAG}\label{lem_bound-sn-DAG}
    For every plane acyclic digraph $D,$ we have $$\sum_{F\in F(D)} (\lt(F) - 2) \leq 2 |\cT(D)| - 4,$$
    with $\cT(D)$ being the set of terminals of $D$ and  $\lt(F)$ being the number of local terminals around $F.$ 
\end{restatable}

\begin{proof}
    Given a plane DAG $D$ and  a vertex $v\in V(D),$ let $\overline{\lt}(v)$ denote the number of angles at $v$ that are not local terminals. As any of the $2|A(D)|$ angles of $D$ is either a local terminal for some $F$ or  a non-local terminal at some $v,$ we have that:
    $$\sum_{v\in V(D)} \overline{\lt}(v) + \sum_{F\in F(D)} \lt(F) = 2|A(D)|.$$
    Note that $\overline{\lt}(v)=0$ if and only if $v$ is a terminal (i.e. $v\in \cT(D)$). Otherwise, $\overline{\lt}(v)\ge 2.$
     \begin{align*}
        \sum_{F\in F(D)} (\lt(F) - 2) 
         &\le \sum_{F\in F(D)} (\lt(F) - 2) + \sum_{v\in V(D)} \max(0,(\overline{\lt}(v) - 2)) \\
         &\le \left( \sum_{F\in F(D)} \lt(F) \right) - 2|F(D)| + \left( \sum_{v\in V(D)} \overline{\lt}(v) \right) - 2|V(D)| + 2|\cT(D) | \\
         &\le 2(|A(D)| - |F(D)| - |V(D)|) + 2 |\cT(D) | \\
         &\le 2 |\cT(D)| - 4
     \end{align*}
    The last line follows by Euler's formula.
\end{proof}

The above implies that the number of alternating faces is bounded in terms of $|\cT(D)|.$ An even stronger consequence of \Cref{lem_bound-sn-DAG} is a bound on the number of local terminals across all alternating faces.

\lemboundsumlt*
\begin{proof}
    Let us first observe that for any solution $X$ and  any terminal $t\in \cT(D),$ there is an arc $e\in X$ incident to $t.$ Otherwise $t$ would remain a terminal and $D+X$ would not be strong. This implies that if $k<\frac12 |\cT(D)|,$ then $(D,k)$ is a \no-instance. In the following we thus assume that $k\ge \frac12 |\cT(D)|.$
    
    Now, since $\lt(F) \geq 4$ for any face $F \in \AF(D),$ \Cref{lem_bound-sn-DAG} yields $|\AF(D)|\le |\cT(D)|-2.$
    Then, all terms corresponding to simple faces in $\sum_{F\in F(D)} (\lt(F) - 2)$ are null, so $\sum_{\AF\in F(D)} \lt(F)  =  \sum_{F\in F(D)} (\lt(F) - 2) + 2 |\AF(D)|.$
    Therefore, combining \Cref{lem_bound-sn-DAG} with the bound on $|\AF(D)|$ yields $\sum_{F\in \AF(D)} \lt(F) \leq 4 |\cT(D)| - 8 < 8k.$
\end{proof}

\end{document}